\documentclass[twoside, 10pt]{amsart}

\usepackage[T1]{fontenc}
\usepackage[lf]{Baskervaldx} % lining figures
\usepackage[bigdelims,vvarbb]{newtxmath} % math italic letters from Nimbus Roman
\usepackage[cal=boondoxo]{mathalfa} % mathcal from STIX, unslanted a bit

\usepackage{latexsym}
\usepackage{amssymb}
\usepackage{amsmath}
\usepackage{amsthm}
\usepackage{amscd}
\usepackage{graphicx}
\usepackage[all]{xy}
\usepackage{a4wide}
\usepackage{multirow}
\usepackage{color}
\usepackage[colorlinks,final,hyperindex]{hyperref}
\usepackage[noabbrev,capitalize]{cleveref}
\usepackage{tikz}
\usepackage{tikz-cd}
\tikzset{math3d/.style=
    {x= {(-0.353cm,-0.353cm)}, z={(0cm,1cm)},y={(1cm,0cm)}}}
\tikzset{JLL3d/.style=
    {x= {(0.4cm,-0.2cm)}, z={(0cm,1cm)},y={(-1cm,0cm)}}}

\definecolor{Chocolat}{rgb}{0.36, 0.2, 0.09}
\definecolor{BleuTresFonce}{rgb}{0.215, 0.215, 0.36}
\hypersetup{citecolor=BleuTresFonce, linkcolor=Chocolat}

\newtheorem{proposition}{Proposition}
\newtheorem{theorem}{Theorem}
\newtheorem*{problem}{Problem}
\newtheorem{corollary}{Corollary}
\newtheorem{definition}{Definition}
\newtheorem{remark}{\sc Remark}
\newtheorem{lemma}{Lemma}

\newtheorem{example}{\sc Example}

\newcommand{\RR}{\mathbb{R}}
\newcommand{\ZZ}{\mathbb{Z}}
\newcommand{\ZP}{\mathbb{Z}_{>0}}
\newcommand{\K}{{K}}
\newcommand{\La}{\mathcal{L}}
\newcommand{\Tam}[1]{\mathrm{PBT}_{#1}}
\newcommand{\PT}[1]{\mathrm{PT}_{#1}}
\newcommand{\sC}{\mathcal{C}}
\newcommand{\PolySub}{\mathsf{Poly}}
\renewcommand{\Top}{\mathsf{Top}}
\newcommand{\CW}{\mathsf{CW}}
\newcommand{\sF}{\mathcal{F}}
\DeclareMathOperator{\tp}{top}
\DeclareMathOperator{\bm}{bot}
\DeclareMathOperator{\pr}{pr}
\DeclareMathOperator{\conv}{conv}
\DeclareMathOperator{\diam}{diam}
\newcommand{\ang}[1]{\langle #1\rangle}
\newcommand{\No}{\mathcal{N}}
\newcommand{\Ve}{\mathcal{V}}
\newcommand{\Bool}[1]{\{0, 1\}^{#1}}
\newcommand{\sk}{\mathrm{sk}}

\newcommand{\tr}{\mathrm{tr}}
\newcommand{\id}{\mathrm{id}}

\DeclareMathOperator{\im}{\mathrm{Im}}
\newcommand{\abs}[1]{\lvert #1\rvert}
\DeclareMathOperator{\codim}{codim}

\newcommand{\TreeLL}
{\vcenter{\hbox{
\begin{tikzpicture}[yscale=0.2,xscale=0.2]
\draw[thick] (0,-1)--(0,0) -- (-3,3);
\draw[thick] (-1,1)--(0,2) ;
\draw[thick] (-2,2)--(-1,3) ;
\draw[thick] (0,0)--(1,1) ;
\draw [fill] (0,-1) circle [radius=0.035];
\draw [fill] (-2,2) circle [radius=0.035];
\draw [fill] (0,2) circle [radius=0.035];
\draw [fill] (1,1) circle [radius=0.035];
\draw [fill] (-3,3) circle [radius=0.035];
\draw [fill] (-1,3) circle [radius=0.035];
\end{tikzpicture}}}}
\newcommand{\TreeLR}
{\vcenter{\hbox{
\begin{tikzpicture}[yscale=0.2,xscale=0.2]
\draw[thick] (0,-1)--(0,0) -- (-2,2);
\draw[thick] (-1,1)--(1,3) ;
\draw[thick] (0,0)--(1,1) ;
\draw[thick] (0,2)--(-1,3) ;
\draw [fill] (0,-1) circle [radius=0.035];
\draw [fill] (-2,2) circle [radius=0.035];
\draw [fill] (0,2) circle [radius=0.035];
\draw [fill] (1,1) circle [radius=0.035];
\draw [fill] (-1,3) circle [radius=0.035];
\draw [fill] (1,3) circle [radius=0.035];
\end{tikzpicture}}}}
\newcommand{\TreeRR}
{\vcenter{\hbox{
\begin{tikzpicture}[yscale=0.2,xscale=0.2]
\draw[thick] (0,-1)--(0,0) -- (3,3);
\draw[thick] (1,1)--(0,2) ;
\draw[thick] (2,2)--(1,3) ;
\draw[thick] (0,0)--(-1,1) ;
\draw [fill] (0,-1) circle [radius=0.035];
\draw [fill] (2,2) circle [radius=0.035];
\draw [fill] (0,2) circle [radius=0.035];
\draw [fill] (-1,1) circle [radius=0.035];
\draw [fill] (3,3) circle [radius=0.035];
\draw [fill] (1,3) circle [radius=0.035];
\end{tikzpicture}}}}
\newcommand{\TreeRL}
{\vcenter{\hbox{
\begin{tikzpicture}[yscale=0.2,xscale=0.2]
\draw[thick] (0,-1)--(0,0) -- (2,2);
\draw[thick] (1,1)--(-1,3) ;
\draw[thick] (0,0)--(-1,1) ;
\draw[thick] (0,2)--(1,3) ;
\draw [fill] (0,-1) circle [radius=0.035];
\draw [fill] (2,2) circle [radius=0.035];
\draw [fill] (0,2) circle [radius=0.035];
\draw [fill] (-1,1) circle [radius=0.035];
\draw [fill] (1,3) circle [radius=0.035];
\draw [fill] (-1,3) circle [radius=0.035];
\end{tikzpicture}}}}
\newcommand{\TreeCC}
{\vcenter{\hbox{
\begin{tikzpicture}[yscale=0.2,xscale=0.2]
\draw[thick] (0,-1)--(0,0) -- (-2.5,2.5);
\draw[thick] (-1.5,1.5)--(-0.5,2.5) ;
\draw[thick] (1.5,1.5)--(0.5,2.5) ;
\draw[thick] (0,0)--(2.5,2.5) ;
\draw [fill] (0,-1) circle [radius=0.035];
\draw [fill] (-2.5,2.5) circle [radius=0.035];
\draw [fill] (-0.5,2.5) circle [radius=0.035];
\draw [fill] (0.5,2.5) circle [radius=0.035];
\draw [fill] (2.5,2.5) circle [radius=0.035];
\end{tikzpicture}}}}
\newcommand{\TreeAL}
{\vcenter{\hbox{
\begin{tikzpicture}[yscale=0.2,xscale=0.2]
\draw[thick] (0,-1)--(0,0) -- (-2,2);
\draw[thick] (-1,1)--(0,2) ;
\draw[thick] (0,0)--(1,1) ;
\draw[thick] (-1,1)--(-1,2) ;
\draw [fill] (0,-1) circle [radius=0.035];
\draw [fill] (-2,2) circle [radius=0.035];
\draw [fill] (0,2) circle [radius=0.035];
\draw [fill] (1,1) circle [radius=0.035];
\draw [fill] (-1,2) circle [radius=0.035];
\end{tikzpicture}}}}
\newcommand{\TreeRA}
{\vcenter{\hbox{
\begin{tikzpicture}[yscale=0.2,xscale=0.2]
\draw[thick] (0,-1)--(0,0) -- (2,2);
\draw[thick] (1,1)--(0,2) ;
\draw[thick] (0,0)--(-1,1) ;
\draw[thick] (1,1)--(1,2) ;
\draw [fill] (0,-1) circle [radius=0.035];
\draw [fill] (2,2) circle [radius=0.035];
\draw [fill] (0,2) circle [radius=0.035];
\draw [fill] (-1,1) circle [radius=0.035];
\draw [fill] (1,2) circle [radius=0.035];
\end{tikzpicture}}}}
\newcommand{\TreeAR}
{\vcenter{\hbox{
\begin{tikzpicture}[yscale=0.2,xscale=0.2]
\draw[thick] (0,-1)--(0,0) -- (2,2);
\draw[thick] (1,1)--(0,2) ;
\draw[thick] (0,0)--(-1,1) ;
\draw[thick] (0,0)--(0,1) ;
\draw [fill] (0,-1) circle [radius=0.035];
\draw [fill] (2,2) circle [radius=0.035];
\draw [fill] (0,2) circle [radius=0.035];
\draw [fill] (-1,1) circle [radius=0.035];
\draw [fill] (0,1) circle [radius=0.035];
\end{tikzpicture}}}}
\newcommand{\TreeLA}
{\vcenter{\hbox{
\begin{tikzpicture}[yscale=0.2,xscale=0.2]
\draw[thick] (0,-1)--(0,0) -- (-2,2);
\draw[thick] (-1,1)--(0,2) ;
\draw[thick] (0,0)--(1,1) ;
\draw[thick] (0,0)--(0,1) ;
\draw [fill] (0,-1) circle [radius=0.035];
\draw [fill] (-2,2) circle [radius=0.035];
\draw [fill] (0,2) circle [radius=0.035];
\draw [fill] (1,1) circle [radius=0.035];
\draw [fill] (0,1) circle [radius=0.035];
\end{tikzpicture}}}}
\newcommand{\TreeCA}
{\vcenter{\hbox{
\begin{tikzpicture}[yscale=0.2,xscale=0.2]
\draw[thick] (0,-1)--(0,0) -- (-1,1);
\draw[thick] (0,1.5)--(1,2.5) ;
\draw[thick] (0,1.5)--(-1,2.5) ;
\draw[thick] (0,0)--(1,1) ;
\draw[thick] (0,0)--(0,1.5) ;
\draw [fill] (0,-1) circle [radius=0.035];
\draw [fill] (-1,1) circle [radius=0.035];
\draw [fill] (1,2.5) circle [radius=0.035];
\draw [fill] (-1,2.5) circle [radius=0.035];
\draw [fill] (1,1) circle [radius=0.035];
\draw [fill] (0,1) circle [radius=0.035];
\end{tikzpicture}}}}
\newcommand{\TreeC}
{\vcenter{\hbox{
\begin{tikzpicture}[yscale=0.2,xscale=0.2]
\draw[thick] (0,-1.5)--(0,0);
\draw[thick] (0,0)--(1.5,1.5) ;
\draw[thick] (0,0)--(0.5,1.5) ;
\draw[thick] (0,0)--(-0.5,1.5) ;
\draw[thick] (0,0)--(-1.5,1.5) ;
\draw [fill] (0,-1.5) circle [radius=0.035];
\draw [fill] (1.5,1.5) circle [radius=0.035];
\draw [fill] (0.5,1.5) circle [radius=0.035];
\draw [fill] (-1.5,1.5) circle [radius=0.035];
\draw [fill] (-0.5,1.5) circle [radius=0.035];
\end{tikzpicture}}}}

\title{The diagonal of the associahedra}

\author{Naruki Masuda}
\address{Johns Hopkins University, 
Department of Mathematics, 
3400 N. Charles Street,  
Baltimore, MD 21218, USA.}
\email{nmasuda2@jhu.edu}

\author{Hugh Thomas}
\address{D\'epartement de math\'ematiques,
  Universit\'e du Qu\'ebec \`a Montr\'eal,
Local PK-5151, 
201, Avenue du Pr\'esident-Kennedy, 
Montr\'eal, Canada.}
\email{thomas.hugh\_r@uqam.ca}

\author{Andy Tonks}
\address{Department of Mathematics, 
University of Leicester,  
University Road,  
Leicester LE1 7RH, 
United Kingdom.}
\email{apt12@leicester.ac.uk}

\author{Bruno Vallette}
\address{Laboratoire Analyse, G\'eom\'etrie et Applications, Universit\'e Paris Nord 13, Sorbonne Paris Cit\'e, CNRS, UMR 7539, 93430 Villetaneuse, France.}
\email{vallette@math.univ-paris13.fr}

\date{\today}



\thanks{N.M. and B.V. were supported by the Institut Universitaire de France and by the grant ANR-14-CE25-0008-01 project SAT. H.T. was supported by an NSERC Discovery Grant and the Canada Research Chairs program. A.T was partially supported by the grant MTM2016-76453-C2-2-P (AEI/FEDER, UE)}

\begin{document}

\begin{abstract}
This paper introduces a new method to solve the problem of the approximation of the diagonal for face-coherent families of  polytopes. We recover the classical cases of the simplices and the cubes and we solve it for the associahedra, also known as Stasheff polytopes. We show that it satisfies an easy-to-state cellular formula. For the first time, we endow a family of realizations of the associahedra (the Loday realizations) with a topological and cellular operad structure; it is shown to be compatible with the diagonal maps. 
\end{abstract}

\maketitle

\setcounter{tocdepth}{1}
\tableofcontents

\section*{Introduction}

The present paper has three goals: to introduce new machinery to solve the problem of the approximation of the diagonal of face-coherent families of  polytopes (\cref{Sec:CanDia}), to give a complete proof for the case of the associahedra (\cref{thm:MainOperad}) and, last but not least, to popularize the resulting \emph{magical formula} (\cref{thm:MagicFormula}) to facilitate its application in other domains.

\bigskip

The problem of the approximation of the diagonal of the associahedra lies at the crossroads of three clusters of domains. 
There are first the mathematicians who will apply it in their work: 
to compute the homology of fibered spaces  in algebraic topology \cite{Brown59, Proute11}, to construct tensor products of string field theories 
\cite{GaberdielZwiebach97}, or to consider the product of Fukaya $\mathcal{A}_\infty$-categories  in symplectic geometry \cite{Seidel08}. Second, there is the  community of operad theory and homotopical algebra, where the analogous result is known in the differential graded context \cite{SaneblidzeUmble04, MarklShnider06} and expected on the topological level. Third, there are combinatorists and discrete geometers who can appreciate our result conceptually as a new development in the theory of fiber polytopes of Billera--Sturmfels \cite{BilleraSturmfels92}.

\bigskip

The possible ways of iterating a binary product can be encoded, for example, by planar binary trees. Interpreting the associativity relation as an order relation, Dov Tamari introduced in his thesis \cite{Tamari51} a lattice structure on the set of planar binary trees with $n$ leaves, now known as the \emph{Tamari lattice}. These lattices can be realized by polytopes, called the \emph{associahedra}, in the sense that their 1-skeleton is the Hasse diagram of the Tamari lattice. We refer the reader to \cite{CFZ02, GKZ08, Loday04a} for examples and the introduction of \cite{CSZ15} for a comprehensive survey.

\bigskip

For loop spaces, composition fails to be strictly associative due to the different parametrizations, but this failure can be controlled by an infinite sequence of higher homotopies. This was made precise by James D. Stasheff in his thesis \cite{Stasheff63}. He introduced a family of  curvilinear polytopes, called the \emph{Stasheff polytopes}, whose combinatorics coincides with the associahedra. Endowing them with a suitable operad structure, that is, an algebraic way to compose operations of various arities, allowed him to establish a now famous recognition theorem for loop spaces. In Stasheff's theory, what is important is to have 
a family of contractible CW-complexes, endowed with an operad structure, whose face lattice is isomorphic to the lattice of planar trees. 

\bigskip

Stasheff's thesis was a profound breakthrough which  opened the door to the study of homotopical algebra by means of operad-like objects.
It prompted, for instance, the seminal monograph of Boardman--Vogt \cite{BoardmanVogt73} on the homotopy properties of algebraic structures, and the recognition of iterated loop spaces \cite{May72}. In this direction, Peter May  introduced the little disks operads, which play a key role in many domains nowadays. In dimension 1, this gives the `little intervals' operad, a finite dimensional topological operad satisfying Stasheff's theory. Its operad structure is given by scaling a configuration of intervals in order to  insert it into another interval.  

\bigskip

Thus two communities, one working on operad and homotopy theories, the other on combinatorics and discrete geometry, \emph{seem} to share a common object. Until now, however, no operad structure on any family of convex polytopal realizations of the associahedra has appeared in the literature. One was proposed in \cite[Part~II Section~1.6]{MSS} but does not hold as faces cannot be scaled in the same way as little intervals, and  in \cite{AguiarArdila17} the problem was solved up to a notion of `(quasi)-normal equivalence'.

\bigskip

In general, the set-theoretic diagonal of a polytope will fail to be cellular. Therefore, there is  a need to find a \emph{cellular approximation to the diagonal}, that is,  a cellular map homotopic to the diagonal. For a face-coherent family of polytopes, that is to say, a family where each face of each polytope in the family is combinatorially a product of lower-dimensional polytopes from the family, finding a family of diagonals compatible with the combinatorics of faces 
 is a very constrained problem. 
In the case of the first face-coherent family of polytopes, the geometric simplices, such a diagonal map is given by the classical \emph{Alexander--Whitney map} of \cite{EilenbergZilber53, EilenbergMacLane54}. This seminal object in algebraic topology allows one to define the associative cup product on the singular cochains of a topological space. 
(The lack of commutativity of the cup product gives rise to the celebrated Steenrod squares \cite{Steenrod47}.) The next family is given by cubes, for which a coassociative approximation to the diagonal is straightforward, see Jean-Pierre Serre's thesis \cite{Serre51}. 
The associahedra form the face-coherent family of polytopes that comes next in terms of further truncations of the simplices or of combinatorial complexity. 
For this family there was, until now, no known approximation to the diagonal. 
While a face of a simplex or a cube is a simplex or a cube of lower dimension, a face of an associahedron is a \emph{product} of associahedra of lower dimensions; this makes the problem of the approximation of the diagonal more intricate. 

\medskip

\[\begin{array}{|c|c|c|c|c|}
\cline{2-4}
\multicolumn{1}{r|}{}  & \rule{0pt}{10pt} \textrm{Dimension}\ 1 &\textrm{Dimension}\ 2 & \textrm{Dimension}\ 3  \\
\cline{1-4}
\textsc{Simplices}
&
{\vcenter{\hbox{
\begin{tikzpicture}[yscale=0.5,xscale=0.5]
\draw[thick] (-1,0)--(1,0); 
\end{tikzpicture}}}}
 & 
 {\vcenter{\hbox{
\begin{tikzpicture}[scale=0.7]
\draw[thick] (0,1) -- (1,-0.66) -- (-1,-0.66) --(0,1)--cycle;
\draw[fill, opacity=0.12] (0,1) -- (1,-0.66) -- (-1,-0.66) --(0,1);
\end{tikzpicture}}}}
 & 
 \vcenter{\hbox{
\begin{tikzpicture}[math3d, scale=1.4]
\draw[thick, opacity=0.2] (1,0,0) -- (0,1,0); 

\draw[thick]  (1.2,0.8,1) --(0,1,0)-- (1.5,1,0)--(1,0,0)--cycle ;
\draw[thick]  (1.5,1,0) --(1.2,0.8,1) ;

\draw[fill, opacity=0.18] (1,0,0) -- (1.2,0.8,1) -- (1.5,1,0) --cycle ;
\draw[fill, opacity=0.06] (1.2,0.8,1) -- (1.5,1,0) -- (0,1,0) -- (1.2,0.8,1) ;

\draw (1.9,1,0) ;
\draw (0,0,0.65);
\end{tikzpicture}}} 
\\
\cline{1-4}
 %%%%%%%% CUBES 
\textsc{Cubes} &
{\vcenter{\hbox{
\begin{tikzpicture}[scale=0.5]
\draw[thick] (-1,0)--(1,0); 
\end{tikzpicture}}}}
& {\vcenter{\hbox{
\begin{tikzpicture}[scale=0.6]
\draw[thick] (-1,-1)--(1,-1) -- (1,1)--(-1,1)-- cycle;
\draw[fill, opacity=0.12] (-1,-1)--(1,-1) -- (1,1)--(-1,1)-- cycle;
\end{tikzpicture}}}}
 & \vcenter{\hbox{
\begin{tikzpicture}[math3d, scale=0.58]
\draw[thick]  (1,-1,1) --(1,1,1)--(1,1,-1)--(1,-1,-1)--cycle ;
\draw[fill, opacity=0.12]  (1,-1,1) --(1,1,1)--(1,1,-1)--(1,-1,-1)--cycle ;
\draw[thick]  (1,-1,1)--(-1, -1,1)--(-1, 1,1);
\draw[fill, opacity=0.18]  (1,-1,1)--(-1, -1,1)--(-1, 1,1)--(1,1,1)--cycle ;
\draw[thick]  (-1, 1,1)--(-1,1,-1)--(1,1,-1);
\draw[fill, opacity=0.06]  (-1, 1,1)--(1,1,1)--(1,1,-1)--(-1,1,-1)--cycle ;
\draw[thick]  (1,1,1)--(-1,1,1);
\draw[thick]  (-1,-1,1)--(-1,1,1)--(-1,1,-1);
\draw[thick, opacity=0.2]  (-1,-1,1)--(-1,-1,-1)--(-1,1,-1);
\draw[thick, opacity=0.2]  (-1,-1,-1)--(1,-1,-1);
\draw (-1,-1,1.25) ;
\draw (1,1,-1.25) ;
\end{tikzpicture}}}
 \\
\cline{1-4}
\textsc{Associahedra} %%%%%%%% ASSOCIAHEDRA
 &
{\vcenter{\hbox{
\begin{tikzpicture}[yscale=0.5,xscale=0.5]
\draw[thick] (-1,0)--(1,0); 
\end{tikzpicture}}}}
& {\vcenter{\hbox{
\begin{tikzpicture}[xscale=0.35, yscale=0.25]
\draw[thick] (0,3)--(-2,1)--(-2,-1)--(0,-3)--(2.5,0)--cycle;
\draw[fill, opacity=0.12] (0,3)--(-2,1)--(-2,-1)--(0,-3)--(2.5,0)--cycle;
\end{tikzpicture}}}}
 &  {\vcenter{\hbox{\begin{tikzpicture}[scale=0.22, JLL3d]
\draw[thick] (3,4,3)--(3,0,3)--(3,0,-1)--(3,2,-1)--(3,4,1)--cycle;
\draw[thick] (-3,4,3)--(-3,-4,3)--(-1,-4,3);
\draw[thick] (-1,-4,3)--(3,0,3);
\draw[thick] (-3,-4,3)--(-1,-4,3)--(-1,-4,-3)--(1,-2,-3);
\draw[thick] (1,-2,-3)--(3,0,-1);
\draw[thick] (-1,-4,-3)--(-1,-4,3);
\draw[thick] (-1,-4,-3)--(1,-2,-3)--(1,0,-3)--(-3,0,-3); 
\draw[thick, opacity=0.2] (-3,4,1)--(-3,0,-3)--(-3,-4,-3)--(-3,-4,3)--(-3,4,3);
\draw[thick, opacity=0.2] (-3,-4,-3)--(-1,-4,-3);  
\draw[thick] (-3,4,3)--(3,4,3)--(3,4,1);
\draw[thick] (-3,-4,3)--(-3,4,3)--(-3,4,1)--(3,4,1);
\draw[thick] (-3,4,3)--(-3,4,1)--(-3,0,-3);
\draw[thick] (1,0,-3)--(3,2,-1);
\draw[fill, opacity=0.18] (-3,4,1)--(-3,4,3)--(3,4,3)--(3,4,1)--cycle;
\draw[fill, opacity=0.18] (-1,-4,3)--(3,0,3)--(3,4,3)--(-3,4,3)--(-3,-4,3)--cycle; 
\draw[fill, opacity=0.18] (-3,0,-3)--(-3,4,1)--(3,4,1)--(3,2,-1)--(1,0,-3)--cycle;
\draw[fill, opacity=0.06] (1,-2,-3)--(3,0,-1)--(3,0,3)--(-1,-4,3)--(-1,-4,-3)--cycle;
\draw[fill, opacity=0.06] (1,-2,-3)--(1,0,-3)--(3,2,-1)--(3,0,-1)--cycle; 
\draw[fill, opacity=0.12] (3,4,3)--(3,0,3)--(3,0,-1)--(3,2,-1)--(3,4,1)--cycle;
\draw (0,0,-3.9);
\draw (0,0,4.5);
\end{tikzpicture}}}}
\\
\cline{1-4}
\end{array}\]

\medskip

The two-fold main result of this paper is: an explicit operad structure on the Loday realizations of the associahedra together with a compatible approximation to the diagonal (\cref{thm:MainOperad}). To accomplish this, we first consider a geometric definition (\cref{def:Diag}) 
for a diagonal map of a polytope suitably oriented by a vector in general position. 
Such an approach comes from the theory of fiber polytopes of Billera--Sturmfels \cite{BilleraSturmfels92}, after Gel'fand--Kapranov--Zelevinsky's theory of secondary polytopes \cite{GKZ08}. 
In order to define the operad structure, we resolve the issue that a face of an associahedron may not be affinely equivalent to a product of the lower-dimensional associahedra, by introducing a notion of Loday realization with arbitrary \emph{weights}.
Since we are looking for an operad structure compatible with the diagonal,
we define it using the diagonal, without loss of generality.
(Notice that the aforementioned coherence for the diagonal maps with respect to the combinatorics of faces amounts precisely to this compatibility with the operad structure.) In the end, this provides the literature with the first  object common to both of the aforementioned communities, providing discrete geometers an extra algebraic structure on realizations of the associahedra, and  homotopy theorists a polytopal (and thus finite) cellular topological $\mathcal{A}_\infty$-operad that recognizes loop spaces. 

\bigskip

Throughout the paper, there is a dichotomy between pointwise and cellular formulas. 
In order to investigate their relationship and to make precise the various face-coherent properties, we introduce a meaningful notion of \emph{category of polytopes with subdivision} which suits our needs. 
Since the definition of the diagonal maps comes from the theory of fiber polytopes, we get an induced polytopal subdivision of the associahedra. In fact, we prove a \emph{magical formula} for it, in the words of Jean-Louis Loday: it is made up of the pairs of cells of matching dimensions and comparable under the Tamari order (\cref{thm:MagicFormula}).  This recovers the differential graded formula of \cite{SaneblidzeUmble04, MarklShnider06}.

\bigskip

The new methods introduced in the present paper should allow one to attack the problem of the approximation of the diagonal for other families of polytopes, such as the ones coming from the theory of operads. 
Our first subsequent plan
is to treat  the case of the multiplihedra \cite{Stasheff70} since these  polytopes encode the notion of $\infty$-morphisms between $\mathcal{A}_\infty$-algebras. This will provide us with the construction of the tensor products of $\mathcal{A}_\infty$-categories, which is needed in symplectic geometry. 
There are then the cases of the cyclohedra, permutoassociahedra, nestohedra, hypergraph polytopes, etc. These would give rise, for instance, to a tensor product construction for homotopy operads. Another relevant question the present approach allows one to study is ``what kind of monoidal $\infty$-category structure does the collection of $\mathcal{A}_\infty$-algebras admit?'' 
In \cite{MarklShnider06}, it is proven that the differential graded diagonal cannot be coassociative. We expect that the fiber polytope method can measure the failure of this coassociativity and a useful formulation for the attacking this problem.
\medskip

\subsection*{Layout} 
The paper is organized as follows. The first section recalls the main relevant notions, introduces the new category of polytopes in which we work. \cref{Sec:CanDia} gives a canonical definition of the diagonal map for positively oriented polytopes and states its cellular properties. In the third section, we endow the family of Loday realizations of the associahedra with a (nonsymmetric) operad structure compatible with the diagonal maps. \cref{sec:MagicalFormula} states and proves the magical cellular formula for the diagonal map of the associahedra. 
\medskip

\subsection*{Conventions} We use the conventions and notations of \cite{Ziegler95} for convex polytopes and the ones of  \cite{LodayVallette12} for operads. We consider only \emph{convex} polytopes whose vertices are the extremal points; they are equivalently defined as the intersection of finitely many half-spaces or as the convex hull of a finite set of points. We simply call them \emph{polytopes}; we denote their sets of vertices by $\Ve(P)$, their face lattices by $\La(P)$, and their normal fans by $\No_P$.
\medskip

\subsection*{Acknowledgements} The first author wishes to thank the International Liaison Office of School of Science of the University of Tokyo for their support through the Study and Visit Abroad Program. 
\newpage

\section{The approximation of the diagonal of the associahedra} 
\subsection{Planar trees, Tamari lattices, and associahedra}

We consider the set $\Tam{n}$ of planar binary (rooted) trees with $n$ leaves, for $n\geq 1$. We read planar binary trees according to gravity, that is from the leaves to the root. The edges of a rooted tree are of three types: the internal edges are bounded by two vertices, the leaves lie at the top and the root at the bottom. 

\begin{definition}[Tamari order \cite{Tamari51}]
The \emph{Tamari order} is the partial order, denoted by $<$, on the set of planar binary trees generated by the following  covering relation

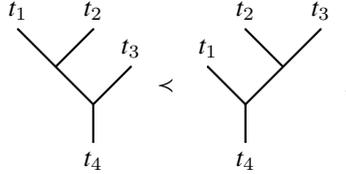
\[{\vcenter{\hbox{
\begin{tikzpicture}[yscale=0.5,xscale=0.5]
\draw[thick] (0,-1)--(0,0) -- (-2,2);
\draw[thick] (-1,1)--(0,2) ;
\draw[thick] (0,0)--(1,1) ;
\draw [fill] (0,-1) circle [radius=0.014];
\draw [fill] (-2,2) circle [radius=0.014];
\draw [fill] (1,1) circle [radius=0.014];
\draw [fill] (0,2) circle [radius=0.014];

\draw (-2,2) node[above] {$t_1$}; 
\draw (0,2) node[above] {$t_2$}; 
\draw (1,1) node[above] {$t_3$}; 
\draw (0,-1) node[below] {$t_4$}; 
\end{tikzpicture}}}}
\prec
{\vcenter{\hbox{
\begin{tikzpicture}[yscale=0.5,xscale=0.5]
\draw[thick] (0,-1)--(0,0) -- (2,2);
\draw[thick] (1,1)--(0,2) ;
\draw[thick] (0,0)--(-1,1) ;
\draw [fill] (0,-1) circle [radius=0.014];
\draw [fill] (2,2) circle [radius=0.014];
\draw [fill] (-1,1) circle [radius=0.014];
\draw [fill] (0,2) circle [radius=0.014];

\draw (2,2) node[above] {$t_3$}; 
\draw (0,2) node[above] {$t_2$}; 
\draw (-1,1) node[above] {$t_1$}; 
\draw (0,-1) node[below] {$t_4$}; 
\end{tikzpicture}}}}\ ,
\]
where $t_i$, for $1\leq i\leq 4$, are planar binary trees. 
\end{definition}

For every $n\geq 1$, this forms a lattice $(\Tam{n}, <)$.  The
\emph{right-leaning} leaves or internal edges are the ones of type $\searrow$ and 
the \emph{left-leaning} leaves or internal edges are the ones of type $\swarrow$. So two trees satisfy $s\leq t$ if and only if one goes from $s$ to $t$ by switching 
 pairs of successive left and right-leaning  edges to a pair of successive right and left-leaning edges.

\begin{figure}[h]
\[\vcenter{\hbox{
\begin{tikzcd}[column sep=0.7cm, row sep=0.2cm]
& \TreeLL\arrow[rrdd, thin]\arrow[ld, thin]&& \\
\TreeLR\arrow[dd, thin]& && \\
& && \TreeCC \arrow[ldld, thin]\\
\TreeRL\arrow[rd, thin]& && \\
&\TreeRR &&
\end{tikzcd}
}}\]
\caption{The Tamari lattice $(\mathrm{PBT}_{4}, <)$ with minimum at the top.}
\end{figure}
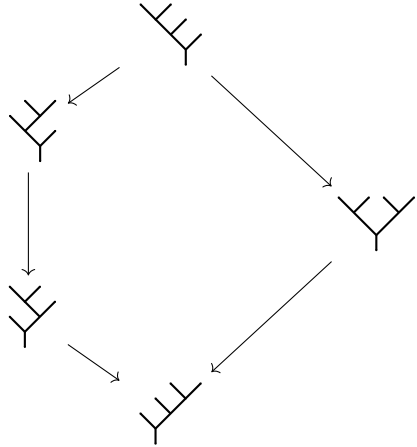

We also consider the  set $\PT{n}$ of all planar trees with $n$ leaves, for $n\geq 1$. 
Each of them forms a lattice (after adjoining a minimum)  under the following partial order: a planar tree $s$ is less than a planar tree $t$, denoted $s\subset t$, 
if $t$ can be obtained from $s$ by a sequence of  edge contractions. 

\begin{definition}[Associahedra]
For any $n\geq 2$, an \emph{$(n-2)$-dimensional associahedron} is a  polytope whose face lattice is isomorphic to the lattice of planar trees with $n$ leaves. 
\end{definition}

\begin{figure}[h]
\[
\begin{tikzpicture}[scale=0.4]
\draw[thick] (-5,0)--(0,4)--(5,0)--(2.5,-4)--(-2.5,-4)--cycle;
\draw[fill, opacity=0.12] (-5,0)--(0,4)--(5,0)--(2.5,-4)--(-2.5,-4)--cycle;
\draw (-5,0) node[left] {$\TreeLL$};
\draw (5,0) node[right] {$\TreeRR$};
\draw (0,4) node[above] {$\TreeCC$};
\draw (-2.5,-4) node[below left] {$\TreeLR$};
\draw (2.5,-4) node[below right] {$\TreeRL$};
\draw (-3.8,-2.6) node[left] {$\TreeAL$};
\draw (3.7,-2.6) node[right] {$\TreeRA$};
\draw (-2.3,1.8) node[above left] {$\TreeLA$};
\draw (2.3,1.8) node[above right] {$\TreeAR$};
\draw (0,-4.2) node[below] {$\TreeCA$};
\draw (0,-0.4) node {$\TreeC$};
\end{tikzpicture}
\]
\caption{A 2-dimensional associahedron.}
\end{figure}
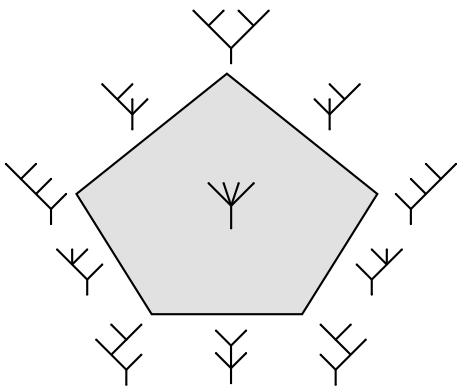

The codimension of a face is equal to the number of internal edges of the corresponding planar tree. 
The $1$-skeleton of an associahedron gives the Hasse diagram of the Tamari lattice. 

\medskip

The operation of grafting a tree $t$ at the $i^{\rm th}$-leaf of a tree $s$ is denoted by $s\circ_i t$. 
These maps endow the collections of planar (binary) trees with a non-symmetric operad structure. 
We denote the corolla with $n$ leaves by $c_n$, i.e. the tree with one vertex and no internal edge. The facets of an $(n-2)$-dimensional associahedron are labelled by 
the two-vertex planar trees $c_{p+1+r} \circ_{p+1} c_q$, for $p+q+r=n$ with $2\leq q\leq n-1$.

\subsection{Loday realizations of the associahedra}

An example of realization of the associahedra can be  given as follows; it is a weighted generalization of the one given by Jean-Louis Loday in \cite{Loday04a}. Notice that Loday realizations produced as convex hulls are the same  polytopes as Shnider--Stasheff \cite{StasheffShnider97} produced by intersections of half-spaces, or equivalently by truncations of standard simplices. 

\begin{definition}[Weighted planar binary tree]
A \emph{weighted planar binary tree} is a pair $(t, \omega)$ made up of a planar binary tree $t\in \Tam{n}$ with $n$ leaves having some weight $\omega= (\omega_1, \ldots, \omega_n) \in \ZP^n$. We call 
$\omega$ the \emph{weight} and
$n$ the \emph{arity} of the tree $t$ or the \emph{length} of the weight $\omega$.
\end{definition}

Let $(t, \omega)$ be a weighted planar binary tree with $n$ leaves. We order its $n-1$ vertices from left to right. At the $i^{\rm th}$ vertex, we consider the sum $\alpha_i$ of the weights of the leaves supported by its left input and 
 the sum $\beta_i$ of the weights of the leaves supported by its right input. 
Multiplying these two numbers, we consider the associated string which defines the following point: 
\[M(t, \omega) \coloneqq \big(\alpha_1\beta_1, \alpha_2\beta_2, \ldots, \alpha_{n-1}\beta_{n-1}\big)\in \RR^{n-1}\ . \]

\begin{definition}[Loday Realization]
 The \emph{Loday realization of weight $\omega$} is the  polytope
\[\K_\omega \coloneqq \conv \big\{M(t, \omega)\mid t\in \Tam{n} \big\}\subset \RR^{n-1}\ .\]
\end{definition}

The Loday realization associated to the standard weight $(1, \ldots, 1)$ is simply denoted by $\K_n$.
By convention, we define the polytope $K_\omega$, with weight $\omega=(\omega_1)$ of length $1$,  to be made up of one point labelled by the trivial tree $|$. 
\begin{figure}[h]
\[
\begin{tikzpicture}[scale=0.7, JLL3d]
\draw[->] (0,4,-2.5)--(1.5,4,-2.5) node[right] {$\vec e_1$};
\draw[->] (0,4,-2.5)--(0,5,-2.5) node[left] {$\vec e_2$};
\draw[->] (0,4,-2.5)--(0,4,-1.5) node[above] {$\vec e_3$};

\draw[thick] (3,4,3)--(3,0,3)--(3,0,-1)--(3,2,-1)--(3,4,1)--cycle;
\draw[thick] (-3,4,3)--(-3,-4,3)--(-1,-4,3);
\draw[thick] (-1,-4,3)--(3,0,3);
\draw[thick] (-3,-4,3)--(-1,-4,3)--(-1,-4,-3)--(1,-2,-3);
\draw[thick] (1,-2,-3)--(3,0,-1);
\draw[thick] (-1,-4,-3)--(-1,-4,3);
\draw[thick] (-1,-4,-3)--(1,-2,-3)--(1,0,-3)--(-3,0,-3); 
\draw[thick, opacity=0.2] (-3,4,1)--(-3,0,-3)--(-3,-4,-3)--(-3,-4,3)--(-3,4,3);
\draw[thick, opacity=0.2] (-3,-4,-3)--(-1,-4,-3);  
\draw[thick] (-3,4,3)--(3,4,3)--(3,4,1);
\draw[thick] (-3,-4,3)--(-3,4,3)--(-3,4,1)--(3,4,1);
\draw[thick] (-3,4,3)--(-3,4,1)--(-3,0,-3);
\draw[thick] (1,0,-3)--(3,2,-1);

\draw[fill, opacity=0.18] (-3,4,1)--(-3,4,3)--(3,4,3)--(3,4,1)--cycle;
\draw[fill, opacity=0.18] (-1,-4,3)--(3,0,3)--(3,4,3)--(-3,4,3)--(-3,-4,3)--cycle; 
\draw[fill, opacity=0.18] (-3,0,-3)--(-3,4,1)--(3,4,1)--(3,2,-1)--(1,0,-3)--cycle;
\draw[fill, opacity=0.06] (1,-2,-3)--(3,0,-1)--(3,0,3)--(-1,-4,3)--(-1,-4,-3)--cycle;
\draw[fill, opacity=0.06] (1,-2,-3)--(1,0,-3)--(3,2,-1)--(3,0,-1)--cycle; 
\draw[fill, opacity=0.12] (3,4,3)--(3,0,3)--(3,0,-1)--(3,2,-1)--(3,4,1)--cycle;
\end{tikzpicture}
\]
\caption{The Loday realization $\K_5$, see \cref{prop:OrientationVector} for the definition of the $\vec e_i$.}
\end{figure}
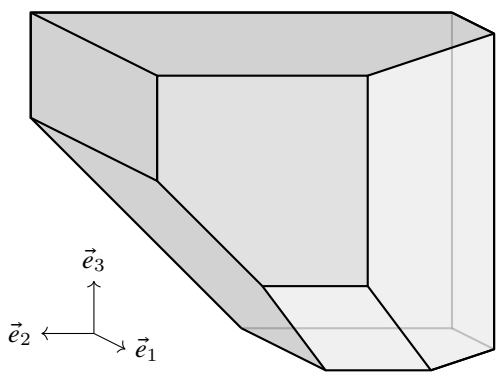

In the sequel, we will need the following key properties of these  polytopes. 

\begin{proposition}\label{prop:PropertiesKLoday}
The Loday realization $\K_\omega$  satisfies the following properties. 
\begin{enumerate}
\item It is contained in the hyperplane with equation
\[
\sum_{i=1}^{n-1} x_i = \sum_{1\leq k< l\leq n} \omega_k \omega_l\ .
\]

\item Let $p+q+r=n$ with $2\leq q\leq n-1$. For any $t\in \Tam{n}$, the point $M(t, \omega)$ is contained in the half-space defined by the inequality
\[
x_{p+1}+\cdots+x_{p+q-1}\geq \sum_{p+1\leq k<l\leq p+q} \omega_k \omega_l\ , 
\]
with equality if and only if the tree $t$ can be decomposed as $t=u\circ_{p+1} v$, where $u\in\Tam{p+1+r}$ and $v\in \Tam{q}$. 

\item The polytope $K_\omega$ is the intersection of the hyperplane of \emph{(1)} and the half-spaces of \emph{(2)}. 

\item The face lattice $(\La(\K_\omega), \subset)$ is equal to the lattice $(\mathrm{PT}_n, \subset)$
of planar trees with $n$ leaves. 

\item Any face of a Loday realization is isomorphic to a product of Loday realizations, via a permutation of coordinates. 
\end{enumerate}
\end{proposition}

\begin{proof}
This proposition is a weighted version of the results of \cite{Loday04a}, except for Point~(5), which actually prompted the introduction of this more general notion. 
\leavevmode
\begin{enumerate}
\item This is straightforward from the definition.
\item This is straightforward too.
\item Let us denote by $P$ the polytope defined by the intersection of the hyperplane of (1) and the half-spaces of (2). 
One can see that the points $M(t, \omega)$, for $t\in \Tam{n}$, are vertices of $P$, since they are defined by a system of $n-1$  independent linear equations: the one of type (1) and $n-2$ of type (2). 
In the other way round, any vertex of $P$ is characterized by a system $n-1$ independent linear equations with the one of type (1) and $n-2$ of type (2). We claim that any pair of equations 
\[\qquad\quad 
x_{p+1}+\cdots+x_{p+q-1}= \sum_{p+1\leq k<l\leq p+q} \omega_k \omega_l 
\quad \text{and} \quad 
x_{p'+1}+\cdots+x_{p'+q'-1}= \sum_{p'+1\leq k<l\leq p'+q'} \omega_k \omega_l 
\]
of type (2) appearing here is such that the corresponding intervals $[p+1, \ldots, p+q-1]$ 
and $[p'+1, \ldots, p'+q'-1]$ are either nested or disjoint. 
If not, we are in the configuration: $p+1<p'+1\leq p+q+1<p'+q'+1$.
Using these equalities and the defining inequalities of $P$, one can get 
\begin{align*}
x_{p'+1}+\cdots+x_{p+q-1}&\leq 
\sum_{p+1\leq k<l\leq p+q} \omega_k \omega_l 
+\sum_{p'+1\leq k<l\leq p'+q'} \omega_k \omega_l 
-\sum_{p+1\leq k<l\leq p'+q'} \omega_k \omega_l \\
&=
\sum_{p'+1\leq k<l\leq p+q} \omega_k \omega_l 
 -
 \sum_{p+1\leq k\leq p'\atop p+q+1\leq l\leq p'+q'} \omega_k \omega_l 
< \sum_{p'+1\leq k<l\leq p+q} \omega_k \omega_l \ ,
\end{align*}
which contradicts the definition of $P$. 
The solution of a system of equations as above, where the defining intervals are nested or disjoint, is a point 
$M(t,\omega)$, with 
$t\in \Tam{n}$.

\item Point~(2) shows that the facets of $\K_\omega$ correspond bijectively to two-vertex planar trees:  the facet labelled by 
$c_{p+1+r}\circ_{p+1} c_q$, for $p+q+r=n$ and $2\leq q \leq n-1$, is the convex hull of the points $M(t, \omega)$ associated to planar binary trees of the form $t=u\circ_{p+1} v$, for $u\in\Tam{p+1+r}$ and $v\in \Tam{q}$. 
Any face of $\K_\omega$ of codimension $k$, for 
$0\leq k \leq n-2$, is defined as an intersection of $k$ facets. The above description of facets gives that the set of faces of codimension $k$ is bijectively labelled by planar trees with $k$ internal edges:  the face corresponding to such a planar tree $t$ is the convex hull of points $M(s, \omega)$, for $s\subset t$. With the top dimensional face labeled by the corolla $c_n$, the statement is proved. 

\item The proof of the above point shows that it is enough to treat the case of the facets. 
Let $p+q+r=n$ with $2\leq q \leq n-1$. We consider the following two weights 
\[
\overline{\omega}\coloneqq (\omega_1, \ldots, \omega_{p}, \omega_{p+1}+\cdots+\omega_{p+q}, \omega_{p+q+1}, \ldots,  \omega_{n})
\quad \text{and} \quad 
\widetilde{\omega}\coloneqq (\omega_{p+1}, \ldots, \omega_{p+q})\ .
\]
The image of $\K_{\overline{\omega}}\times \K_{\widetilde{\omega}}\hookrightarrow \K_\omega$ under  the isomorphism 
\begin{align*}
\begin{array}{rccc}
\Theta\  : &  \RR^{p+r}\times \RR^{q-1} &\xrightarrow{\cong} &\RR^{n-1}\\
&(x_1, \ldots, x_{p+r})\times (y_1, \ldots, y_{q-1}) & \mapsto& 
(x_1, \ldots, x_{p} , y_1, \ldots, y_{q-1}, x_{p+1}, \ldots, x_{p+r})
\end{array}
\end{align*}
is equal to the facet labeled by the planar tree $c_{p+1+r}\circ_{p+1} c_q$\ .
\end{enumerate}
\end{proof}

In other words, Point~(4) shows that the polytopes $\K_\omega$ are realizations of the associahedra. 

\subsection{The category of polytopal subdivisions}
The proposed notions of category of polytopes  present in the literature only allow affine maps, ,which is too restrictive for our purpose: the facets of the Loday realizations of associahedra with standard weights are not affinely equivalent to the product of lower realizations with standard weights. In order to introduce a more suitable category, we begin with the following definition, which extends the classical notion of simplicial complex. 

\begin{definition}[Polytopal complex]
A \emph{polytopal complex} is a finite collection $\sC$ of polytopes of $\RR^n$ satisfying the following conditions:
\begin{enumerate}
\item the empty polytope $\emptyset$ is contained in $\sC$,
\item $P \in \sC$ implies $\La(P) \subset \sC$, 
\item $P, Q\in \sC$ implies $P\cap Q\in \La(P)\cap \La(Q)$.
\end{enumerate}
\end{definition}
Any polytope $P$ gives an example of polytopal complex $\La(P)$ made up of  all its faces.
A subcomplex of a polytopal complex $\sC$ is a subcollection 
$\mathcal{D}\subset \sC$ which  forms a polytopal complex.
The \emph{underlying set} of a collection $\sC$ is given by the union $|\mathcal{C}|\coloneqq\bigcup_{P\in \sC} P\subset \RR^n$.

\begin{definition}[Polytopal subdivision]
A \emph{polytopal subdivision} of a polytope $P$ is a polytopal complex $\sC$ 
whose underlying set $|\sC|$ is equal to $P$. 
\end{definition}

The face poset $\La(\sC)$ of a polytopal complex is defined in the obvious way. 
We say that two polytopal complexes are \emph{combinatorially equivalent} when their face posets are isomorphic.
Now let us introduce the category we will work in.

\begin{definition}[The category $\PolySub$]
The category $\PolySub$ is made up of the following data.
\begin{description}
\item[{\sc Objects}] An object is a $d$-dimensional  polytope $P$ in the $n$-dimensional Euclidian space $\RR^n$, for any $0\leq d\leq n$.
\item[{\sc Morphisms}] A continuous map  $f: P\to Q$ is a morphism when 
it sends  $P$ homeomorphically to the underlying set $|\mathcal{D}|$ of a polytopal subcomplex $\mathcal{D}\subset \La(Q)$ of $Q$ 
such that $f^{-1}(\mathcal D)$ defines a polytopal subdivision of $P$.
\end{description}
\end{definition}

There exists obvious forgetful functors from the category $\PolySub$ to the category $\Top$ of topological spaces with continuous maps and to the category $\CW$ of CW complexes with cellular maps. The latter functor is well-defined since any morphism in $\PolySub$ is automatically cellular.
An isomorphism $P\cong Q$ in this category is a cell-respecting homeomorphism which induces 
a combinatorial equivalence  $\La (P)\cong \La(Q)$.

\begin{lemma}
The category $\PolySub$ endowed with the direct product $\times$ and the zero-dimensional polytope made up of one point is a symmetric monoidal category.
\end{lemma}

\begin{proof}
The verification of the axioms is straightforward. 
\end{proof}

This extra structure allows one to consider operads in the category $\PolySub$. 
Since the  cellular chain functor $\PolySub \to \textsf{dgMod}_{\ZZ}$ is 
 strong symmetric monoidal, it induces a functor from the category of polytopal (non-symmetric) operads to the category of differential graded (non-symmetric) operads. 

\subsection{The approximation of the diagonal of the associahedra}
In the sequel, we solve the following two-fold problem. 

\begin{problem}\label{Prob:cellulardiagonal}\leavevmode
\begin{enumerate}
\item Endow the collection of Loday realizations of the associahedra $\{K_n\}_{n\geqslant 1}$ with a nonsymmetric operad structure in the  category $\PolySub$, whose induced set-theoretical nonsymmetric operad structure on the set of faces coincides with that of planar trees.

\item Endow the collection $\{K_n\}_{n\geqslant 1}$ with diagonal maps $\{\triangle_n: K_n\to K_n\times K_n\}_{n\geq 1}$ which form a morphism of nonsymmetric operads in the category $\PolySub$. 
\end{enumerate}
\end{problem}

\begin{remark}\leavevmode
\begin{enumerate}
\item Even in the category of topological spaces and for any  family of realizations of the associahedra, we do not know any solution to this question in the existing literature. 

\item The compatibility of the diagonal maps with the operad structure amounts precisely to the coherence required by  the approximation of the diagonal maps with respect to sub-faces by  Point~(5) of \cref{prop:PropertiesKLoday}.
\end{enumerate}
\end{remark}

In order to find an operadic cellular approximation to the diagonal of the associahedra, we introduce ideas coming from   the theory of fiber polytopes \cite{BilleraSturmfels92} as follows. 

\section{Canonical diagonal for positively oriented  polytopes}\label{Sec:CanDia}

\subsection{Positively oriented polytopes}

\begin{definition}[Positively oriented polytope]\leavevmode
\begin{enumerate}
\item An \emph{oriented polytope} is a  polytope $P\subset \RR^n$ endowed  with a vector $\vec v\in \RR^n$ such that no edge of $P$ is perpendicular to $\vec v$, see \cref{Fig:PosOr}\ . 
\begin{figure}[h]
\[\vcenter{\hbox{
\begin{tikzpicture}[xscale=0.5, yscale=0.4]
\draw[thick] (0,3)--(-2,1)--(-2,-1)--(0,-3)--(3,0)--cycle;
\draw[->] (4.5,1.5)--(4.5,-1.5);
\draw (4.5,0) node[right] {$\vec v$};
\end{tikzpicture}
}}\]
\caption{An oriented polytope.}\label{Fig:PosOr}
\end{figure}
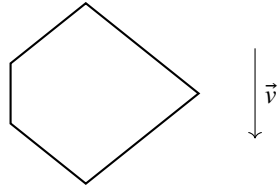

\item A \emph{positively oriented polytope} is an oriented polytope $(P, \vec v)$ such that the intersection polytope 
\[
\big(P\cap \rho_z P, \vec v\big) 
\]
is oriented, where $\rho_z\coloneqq 2z-P$ stands for the reflection with respect to any point $z\in P$, see \cref{Fig:Diag}. 
\end{enumerate}
\end{definition}

The data of an orientation vector induces a poset structure on the set of vertices $\Ve(P)$ of $P$, which is the transitive closure of the relation induced by the oriented edges of the 1-skeleton.

\begin{definition}[Well-oriented realization of the associahedron]
A \emph{well-oriented realization of the associahedron} is a  positively oriented polytope which realizes the associahedron and such that the orientation vector induces the Tamari lattice on the set of vertices. 
\end{definition}

\begin{proposition}\label{prop:OrientationVector}
Let $\omega$ be a weight of length $n$. 
Any vector $\vec v=(v_1, \ldots, v_{n-1})\in \RR^{n-1}$ satisfying $v_1>v_2>\cdots >v_{n-1}$ induces a
well-oriented realization $(\K_\omega, \vec v)$ of the associahedra.
\end{proposition}

\begin{proof}
Let us first prove that such an orientation vector $\vec v$ induces the Tamari lattice. 
Let $s\prec t$ in the Tamari lattice. The corresponding edge in $\K_\omega$ is of the form
\[\overrightarrow{M(s,\omega)M(t,\omega)}=(0, \ldots, 0, x , 0, \ldots, 0, -x, 0, \ldots, 0)\ , \]
 with $x>0$, which implies 
$\Big\langle\overrightarrow{M(s,\omega)M(t,\omega)}, \vec v\Big\rangle= x(v_{j+i}-v_i)>0$\ .
This also proves that $(K_\omega ,\vec v)$ is oriented. 

Let us now prove that $(\K_\omega ,\vec v)$ is positively oriented. 
We denote by $\vec n\coloneqq (1, \ldots, 1)$ and $\vec n_F\coloneqq (0, \ldots, 0,\allowbreak {1}, \ldots, {1}, 0, \ldots, 0)$ the normal vectors of a facet $F$, given by  Points~(1) and (2) of \cref{prop:PropertiesKLoday}.
Since edges of $\K_\omega\cap \rho_z \K_\omega$ 
are one-dimensional intersections of facets of 
$\K_\omega$ or $\rho_z \K_\omega$, their directions $\vec d$ are the unique solutions to a system of equations of type $\langle \vec n, \vec d\rangle=0$ and $\langle \vec n_F, \vec d\rangle=0$, for some set of facets $F$. 
The first equation imposes $\vec d=\sum_{j=1}^{n-2} a_j \vec e_j$, where 
$\vec e_j\coloneqq (0, \ldots, 1, -1, \ldots, 0)$, in which $1$ is in the $j$-th place for $1\leq j\leq n-2$. 
The second equation is equivalent to one of the following three constraints 
$a_{p+1}=0$, $a_{p+q-1}=0$, or $a_{p+q-1}=a_{p+1}$. 
Therefore, $\vec d$ is of the form $\lambda\left(\vec e_{j_1}+\cdots +\vec e_{j_k}\right)$, with $\lambda \in \RR\setminus \{0\}$, and so
$\ang{\vec d, \vec v}\neq 0$. 
\end{proof}

\subsection{Construction of diagonal maps}

Before getting into specific argument on the associahedra, we construct a diagonal map $\triangle: P\to P\times P$ for any positively oriented polytope $(P, \vec v)$. Let $\tp P$ (resp. $\bm P$) denote the top (resp. bottom) vertex of $P$ with respect to the orientation vector $\vec v$.

\begin{definition}[Diagonal of a positively oriented polytope]\label{def:Diag}
The \emph{diagonal} of a positively oriented polytope $(P, \vec v)$ is defined by 
\begin{align*}
\begin{array}{rlcl}
\triangle_{(P,\vec v)}\  : & P &\to  &P\times P\\
&z & \mapsto& 
\bigl(\bm(P\cap \rho_z P),\,  \tp(P\cap \rho_z P)\bigr) \ .
\end{array}
\end{align*}
\end{definition}

\begin{figure}[h]
\[
\begin{tikzpicture}[xscale=0.8, yscale=0.7]
\draw[thick] (1,3)--(-1,1)--(-1,-1)--(1,-3)--(4,0)--cycle;
\draw[thick] (-1,3)--(1,1)--(1,-1)--(-1,-3)--(-4,0)--cycle;

\draw[thick] (0,-2)--(-1,-1)--(-1, 1)--(0, 2)--(1, 1)--(1,-1)--cycle;
\draw[fill, opacity=0.15] (0,-2)--(-1,-1)--(-1, 1)--(0, 2)--(1, 1)--(1,-1)--cycle;

\draw[->] (5,1.2)--(5,-1.2);
\draw (5,0) node[right] {$\vec v$};
\draw (0,0) node {$\bullet$};
\draw (0,2) node {$\bullet$};
\draw (0,-2) node {$\bullet$};

\draw (0,0) node[right] {$z$};
\draw (0,1.9) node[below] {$x$};
\draw (0,-1.9) node[above] {$y$};

\draw (3,-2) node {$P$};
\draw (-3,-2) node {$\rho_z P$};
\end{tikzpicture}
\]
\caption{The diagonal map $\triangle(z)_{(P,\vec v)}=\bigl(\bm(P\cap \rho_z P),\,  \tp(P\cap \rho_z P)\bigr)=(x,y)$\ .}\label{Fig:Diag}
\end{figure}
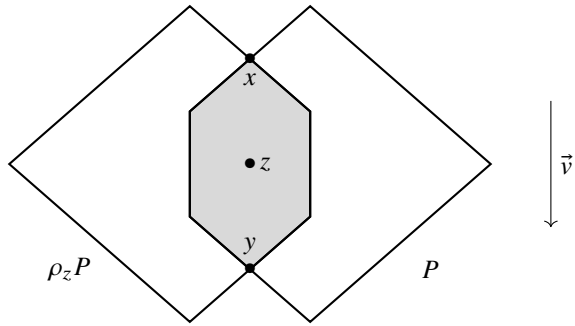

Let $\beta$ be the middle-point map $P\times P\to P; (x,y)\mapsto \frac{x+y}{2}$. 
With the notation  $\pr_i$ for  the projection to the $i$-th coordinate, we have $\pr_1\beta^{-1}(z)=\pr_2\beta^{-1}(z) = P\cap \rho_z P$.
The diagonal $\triangle_{(P,\vec v)}$ is a section of $\beta$ since $P\cap \rho_z P$ is symmetric with respect to the point $z$. 

\medskip

Since the Loday realizations $K_\omega$ of the associahedra are positively oriented when the orientation vector $\vec v$ has decreasing coordinates, the above formula endows them with diagonal maps, which do not depend on the choice  the orienting vector.

\begin{proposition}\label{prop:IndepVect}
Let $\vec v$ and $\vec w$ be two vectors of $\RR^{n-1}$ with decreasing coordinates, then  \[\triangle_{(K_\omega, \vec v)}=\triangle_{(K_\omega, \vec w)}\ ,\]
for every weight $\omega$ of length $n$.
\end{proposition}

\begin{proof}
The argument given in the proof of \cref{prop:OrientationVector} shows that the formula 
$\bigl(\bm(\K_\omega\cap \rho_z \K_\omega),\,  \allowbreak \tp(\K_\omega\cap \rho_z \K_\omega)\bigr)$ produces the same pair for any orientation vector with decreasing coordinates. 
\end{proof}

We denote by $\triangle_\omega : \K_\omega\to \K_\omega \times \K_\omega$ the diagonal map given by any such orientation vector. 
In the case of the standard weight $\omega=(1, \ldots, 1)$, we denote it simply by $\triangle_n : \K_n \to \K_n\times \K_n$. 

\begin{lemma}\label{Lemma:DiagOnFace}
Any face $F$ of a positively oriented polytope $(P,\vec v)$ is positively oriented once equipped with the orientation vector $\vec v$. Moreover, the two  diagonals agree: $\triangle_{(P, \vec v)}(z)=\triangle_{(F, \vec v)}(z)$, for any $z\in F$.
\end{lemma}

\begin{proof}
This follows from the relation 
$P\cap \rho_z P= F\cap \rho_z F$, for any $z\in F$. 
\end{proof}

\subsection{Polytopal subdivision induced by the diagonal}\label{subsec:PolySubDiv}
The above formula for the diagonal $\triangle$ actually defines a morphism in the category $\PolySub$. 
To prove this property, we use some ideas coming from  the theory of fiber polytopes \cite{BilleraSturmfels92}, see also \cite[Chapter~9]{Ziegler95}.

\medskip

Let $\pi:  P\twoheadrightarrow Q$ be an affine projection of polytopes with 
$P\subset \RR^p$ and $Q\subset \RR^q$. 
We denote by $P_y\coloneqq P\cap \pi^{-1}(y)$ the fiber above $y\in Q$.
To any linear form $\psi\in (\RR^p)^{\ast}$, we associate a 
collection 
 $\sF^\psi\subset \La(P)$ as follows. 
We first factor the projection $\pi=\mathrm{pr}\circ \tilde{\pi}$ into the two  maps: 
\[\tilde{\pi}\coloneqq (\pi, \psi): P\twoheadrightarrow \tilde{Q}
\quad \text{and} \quad
 \mathrm{pr}: \tilde{Q}\twoheadrightarrow Q;\, (x, t)\mapsto x\ ,  
\quad \text{where}\quad  \tilde{Q}\coloneqq \big\{(\pi(y), \psi(y))\mid y\in P\big\}\subset \RR^{q+1}\ .
\]

Let $\La^{\downarrow}\big(\tilde Q\big) \subset \La\big(\tilde{Q}\big)$ be the subcomplex of lower faces, i.e. $F\in \La^{\downarrow}\big(\tilde Q\big)$ if and only if any $(y, t)\in F$ satisfies the equation 
$t = \min \psi(P_y)$. Since the preimage of any face by a projection of polytopes is again a face, this defines a collection 
\[\sF^\psi \coloneqq 
 \big\{P\cap  \tilde{\pi}^{-1}(F) \mid F\in \La^{\downarrow}\big(\tilde Q\big) 
\big\}
\subset \La(P)\ .\]
(This is in general not a polytopal complex since it is not stable under faces.)
In the case of the point $Q=\{\ast\}$, the unique face contained in $\sF^\psi$ is by definition 
\[P^\psi\coloneqq \left\{x\in P \mid \psi (x) = \min \psi(P) \right\}\ .\]

\begin{proposition}\label{prop:CoherentSubdivision}
The collection $\sF^\psi \subset \La(P)$  satisfies the following properties.
\begin{enumerate}
\item The polytopal complex $\pi(\sF^\psi)\coloneqq\left\{\pi(F)\mid F\in \sF^\psi \right\}$ is a polytopal subdivision of $Q$.
\item For any $y\in Q$, the fiber satisfies  $(\sF^\psi)_y \coloneqq\pi^{-1}(y) \cap\left| \sF^\psi\right|=\left(P_y\right)^\psi$.
\end{enumerate}
\end{proposition}

\begin{proof}\leavevmode
\begin{enumerate}
\item By definition $\pi\left(\sF^\psi\right) = \tilde\pi\big(\La^{\downarrow}\big(\tilde Q\big)\big)$\ . The right-hand side defines a polytopal subdivision of $Q$, since the restriction of the map  $\pr$ is a linear homeomorphism from a polytopal complex.
\item This is clear from the definition.
\end{enumerate}
\end{proof}

\begin{definition}[Coherent and tight subdivisions]\label{def:Coherent}\leavevmode
\begin{enumerate}
\item A subcollection $\sF\subset\La(P)$ is called a \emph{coherent subdivision of $Q$} when it is of the form $\sF^\psi$ for some $\psi\in (\RR^p)^{\ast}$.
\item A coherent subdivision $\sF$ is called \emph{tight} when, the faces $F$ and $\pi(F)$ have the same dimension, for every $F\in \sF$.
\end{enumerate}
\end{definition}

A coherent subdivision $\sF$ is tight if and only if, for any $y\in Q$, the fiber $\sF_y = (P_y)^\psi$ is a point, by Point~{(2)} of \cref{prop:CoherentSubdivision}. (This is also equivalent to $\sF$ being a polytopal complex.)
Therefore a tight coherent subdivision can be identified with the unique piecewise-linear section of $\pi|_P$ which minimizes $\psi$ in each fiber.
By the Point~{(1)} of \cref{prop:CoherentSubdivision}, this section is a morphism of the category $\PolySub$.

\medskip

We apply these results to the projection 
\[\beta : P\times P\to P ; (x, y)\mapsto \frac{x+y}{2}\ .\]
and to the linear form $\psi(x, y)\coloneqq \ang{x-y, \vec v}$ in order to obtain the following proposition.

\begin{proposition}\label{prop:DiaginPoly}
If $(P, \vec v)$ is a positively oriented polytope, the diagonal map $\triangle_{(P, \vec v)}: P\to P\times P$ is a morphism in the category $\PolySub$. 
\end{proposition}

\begin{proof}
For any $z\in P$, the fiber $\beta^{-1}(z)$ is the set of pairs $(x, y)\in P\times P$ such that $x+y = 2z$. Since the sum of $x$ and $y$ is constant, $\psi(x, y)$ is minimized in the fiber when $\ang{x, \vec v}$ is minimized, or equivalently, $\ang{y, \vec v}$ is maximized. On both coordinates, $\beta^{-1}(z)$ projects down to the intersection $P\cap (2z-P)$, which is oriented by definition. 
So the fiber $\beta^{-1}(z)$ admits a unique minimal element  $\left(\bm(P\cap \rho_z P),\,  \tp(P\cap \rho_z P)\right)$ with respect to $\psi$. In the end, the section defined by the tight coherent subdivision agrees with the definition of the diagonal map $\triangle_{(P, \vec v)}$ given in \cref{def:Diag}.
\end{proof} 

\begin{corollary}\label{cor:DIMENSION}
The image of $\triangle_{(P,\vec v)}$ is contained in $\sk_n (P\times P)$. In particular, if one of two components of $\triangle (z)$ lies in the interior of $P$, then the other component is either $\tp P$ or $\bm P$.
\end{corollary}

\begin{remark}
Notice that the diagonal map $\triangle_{(P,\vec v)}$ is fiber-homotopic to the usual diagonal $x\mapsto (x, x)$\ .
\end{remark}

We denote by $\sF_{(P, \vec v)}$ the corresponding tight coherent subdivision of $P$ and by $\beta\left(\sF_{(P, \vec v)}\right)$ the  polytopal subdivision of $P$. 
In the case of the Loday realizations of the associahedra, \cref{prop:IndepVect} shows that they do not depend on the orientation vector (with decreasing coordinates), so we use the simple notations $\sF_{\omega}$ and $\beta\left(\sF_\omega\right)$.

\medskip

There is a simple way for drawing this polytopal subdivision: one glues together two copies of $P$ along an axis of direction $\vec v$, then one draws all the middle points of pairs of points coming from some choices of a  face on the left-hand side copy and a face on the right-hand side copy, see \cref{Fig:PolySubBary}. In the case of the associahedra, these choices are given by the magical formula of \cref{sec:MagicalFormula}. 
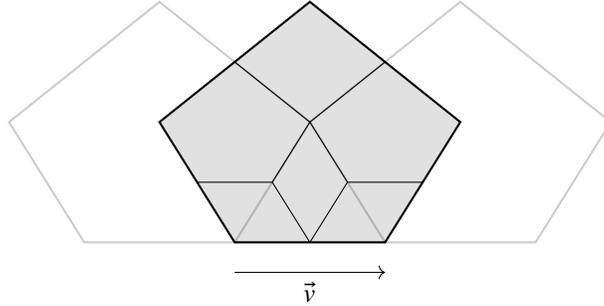
\begin{figure}[h]
\[
\begin{tikzpicture}[xscale=0.5, yscale=0.4]
\draw[thick] (-4,0)--(0,4)--(4,0)--(2,-4)--(-2,-4)--cycle;
\draw[fill, opacity=0.12] (-4,0)--(0,4)--(4,0)--(2,-4)--(-2,-4)--cycle;
\draw (-2,2)--(0,0)--(-1,-2)--(-3,-2);
\draw (2,2)--(0,0)--(1,-2)--(3,-2);
\draw (-1,-2)--(0,-4)--(1,-2);

\draw[thick,opacity=0.2] (-8,0)--(-4,4)--(0,0)--(-2,-4)--(-6,-4)--cycle;
\draw[thick,opacity=0.2] (8,0)--(4,4)--(0,0)--(2,-4)--(6,-4)--cycle;

\draw[->] (-2,-5)--(2,-5);
\draw (0,-5) node[below] {$\vec v$};

\end{tikzpicture}
\]
\caption{Example of polytopal subdivision.}\label{Fig:PolySubBary}
\end{figure}

\begin{example} This approach allows us to recover the classical cases of the simplices and the cubes. 
\begin{enumerate}
\item The classical Alexander--Whitney approximation to the diagonal of simplices $\mathrm{AW}_n : \Delta^n \to \Delta^n\times \Delta^n$ can be recovered with the following  geometric realizations 
\[\Delta^n\coloneqq\conv\big\{(1, \ldots, 1, 0,\ldots, 0)\in \RR^n\big\}=
\big\{(z_1, \ldots , z_n)\in \RR^n\mid 
1 \geq z_1\geq \cdots \geq z_n \geq 0
\big\}\ .
\]
As usual, we denote by $\mathrm{i}$ the point of $\RR^n$ of coordinates $(1, \ldots, 1, 0,\ldots, 0)$, 
where $1$'s appear $i$ times, and we consider the vector  $\vec n\coloneqq (1, \ldots, 1)$ as above. 
The same argument as in the proof of \cref{prop:OrientationVector} shows that $\big(\Delta^n, \vec n \big)$ is positively oriented.
For $z=(z_1, \ldots , z_n)$ satisfying 
$1 \geq z_1\geq \cdots \geq z_i \geq \frac{1}{2} \geq z_{i+1}\geq \cdots \geq z_n \geq 0$, 
One can easily see that  the minimum of $\psi$ is reached by
\[
\triangle_{\big(\Delta^n, \vec n\big)}(z)=
\big(( 
2z_1-1, \ldots, 2z_i-1, 0, \ldots, 0)\, , 
(
1, \ldots, 1, 
2z_{i+1}, \ldots, 2z_n
)\big)\ .
\]
We consider the faces $\Delta^{\{\mathrm{0}, \ldots, \mathrm{i} \}}=
\big\{(z_1, \ldots , z_i, 0, \ldots, 0)\in \RR^n\mid 
1 \geq z_1\geq \cdots \geq z_i \geq 0
\big\}$ and 
$\Delta^{\{\mathrm{i}, \ldots, \mathrm{n} \}}=
\big\{(1, \ldots, 1, z_{i+1}, \ldots , z_n)\in \RR^n\mid 
1 \geq z_{i+1}\geq \cdots \geq z_n \geq 0
\big\}$\ .
The tight coherent subdivision of $\Delta^n$ is  
equal to 
\[\sF_{\big(\Delta^n, \vec n\big)}=
\left\{\Delta^{\{\mathrm{0}, \ldots, \mathrm{i} \}} \times \Delta^{\{\mathrm{i}, \ldots, \mathrm{n} \}} \mid 
0 \leq i \leq n
\right\}\ ,
\]
which recovers the (simplicial) Alexander--Whitney map of \cite{EilenbergZilber53, EilenbergMacLane54}. 

\item The approximation of the diagonal $C^n \to C^n\times C^n$ of the cube $C^n\coloneqq [0,1]^n$ used by Jean-Pierre Serre in \cite{Serre51} can easily be recovered by the present method. First, it is a positively oriented polytope once equipped with the orientation vector 
$\vec n$. Since an $n$-dimensional cube is a product of $n$ intervals, the various formulas are straightforward. 
\end{enumerate}
\end{example}

Notice that these two examples work particularly well because we do not need to stretch the faces and their combinatorial complexity is very limited: any face appearing here is \emph{affinely equivalent} to a lower dimensional polytope of the respective family. These properties do not hold anymore for the Loday realizations of the associahedra; such a difficulty is omnipresent in the rest of this paper. 

\section{Operad structure on Loday realizations}

\subsection{Cellular properties}
\begin{proposition}\label{Prop:NormalFan}
Suppose $P$ and $Q$ are normally equivalent polytopes, 
i.e. with same normal fans $\No_P=\No_Q$. 
If $P$ and $Q$ are well-oriented by the same orientation vector $\vec v$, 
then the tight coherent subdivisions 
$\sF_{(P, \vec v)}$ and $\sF_{(Q, \vec v)}$ are  combinatorially equivalent in a canonical way. 
\end{proposition}

\begin{proof}
Normal equivalence of polytopes induces  combinatorial equivalence $\Phi: \La(P)\cong\La(Q)$.
By the definition of the diagonal map, a pair of points $(x, y)\in P\times P$ is contained in the image $\im \triangle_P$ if and only if it satisfies the following condition:
there exists no vector $\vec w$ and positive number $\varepsilon >0$ with $\ang{\vec v, \vec w}>0$ and $x-\varepsilon \vec w, y+\varepsilon\vec w\in P$.
The latter conditions can be restated in terms of normal cones as follows. Recall that,   
for any subset $C \subset \RR^n$, the polar cone $C^\star$ of $C$ is defined by
\[
C^\star \coloneqq \left\{ x\in \RR^n\mid \forall x\in C, \  \ang{x, y} \leq 0 \right\} .
\]
The polar cone theorem asserts that $C^{\star \star}$ is the smallest closed convex cone which contains $C$. 
By definition, $(P-y)^\star$ is the normal cone $\No_P(G)$ corresponding to the face which satisfies $y\in \mathring G$.
Applying the polar cone theorem to $C= P-y$, we obtain that $(P-y)^{\star\star} = \No_P(G)^\star$ is the set of vectors $\vec w$ such that $y+\varepsilon\vec w \in P$ for some $\varepsilon >0$. 
Therefore if $x\in \mathring F$ and $y\in \mathring G$, the condition for $(x, y)\in \im \triangle_P$ is that 
there exists no vector $\vec w$ such that 
$\ang{\vec v, \vec w}>0$ and $\vec w\in -\No_P(F)^\star\cap \No_P(G)^\star$.
Since this condition depends only on the normal fan, 
the map $\Phi\times \Phi: \La(P\times P)\to \La(Q\times Q)$ induces the canonical combinatorial equivalence $\sF_{(P, \vec v)}\cong \sF_{(Q, \vec v)}$.
\end{proof}

\begin{corollary}\label{Cor:SameDiag}
Let $\omega$ and $\theta$ be two weights of same length.
The two polytopal subdivisions $\beta\left(\sF_\omega \right)$ and $\beta\left(\sF_\theta \right)$ of $\K_\omega$ and $\K_\theta$ respectively are combinatorially equivalent, i.e. labelled by the same pairs of planar trees. 
\end{corollary}

\begin{proof}
This is a direct corollary of Point~(3) of \cref{prop:PropertiesKLoday} and  \cref{Prop:NormalFan}. 
\end{proof}

\cref{prop:IndepVect} and \cref{Cor:SameDiag} show that the type of faces composing the polytopal subdivision 
of the Loday realizations of the associahedra are intrinsic: they depend neither on the orientation vector (with decreasing coordinates) nor on the weight. From now on, we simply denote them by $\sF_n\subset \mathrm{PT}_n\times \mathrm{PT}_n$ and $\beta\left(\sF_n\right)$. Their description will be the subject of the magical formula given in \cref{sec:MagicalFormula}. 

\subsection{Pointwise properties}
We can enhance the above one-to-one correspondence of polytopal subdivisions to the pointwise level using the   isomorphism in the category $\PolySub$. 

\begin{proposition}\label{prop:Transition}
Let $(P, \vec v)$ and $(Q,\vec w)$ be two positively oriented polytopes, with a combinatorial equivalence  $\Phi: \La(P)\xrightarrow{\cong} \La(Q)$.
Suppose that 
tight coherent subdivisions 
$\sF_{(P, \vec v)}$ and $\sF_{(Q, \vec w)}$ are combinatorially equivalent under 
$\Phi\times \Phi$. 
\begin{enumerate}
\item 
There exists a unique continuous map
\[\tr=\tr_P^Q : P\to Q\ ,\]
which extends the restriction $\Ve(P)\to \Ve(Q)$ of $\Phi$ to the set of vertices  
and which commutes with the respective diagonal maps. 
\item The map $\tr$ is an isomorphism in the category $\PolySub$, whose correspondence of faces agrees with $\Phi$.
\end{enumerate}
\end{proposition}

We call this map $\tr$ the \emph{transition map}. 

\begin{proof}\leavevmode
\begin{enumerate}
\item 
In the core of this proof, we use the simple notation $\triangle_P$ for $\triangle_{(P, \vec v)}$ and 
\[\triangle^{(n)}_P\coloneqq \triangle^{2^{n-1}}_P\circ \triangle^{2^{n-2}}_P\circ\cdots\circ\triangle^{2}_P\circ\triangle_P\ : \  
P\to P^{2^{n}}\ ,\]
for its iterations. 
We also denote the averaging map by
\begin{align*}
\begin{array}{rccc}
\beta^{(n)}_P\  : & P^{2^n}&\to &P\\
&(x_1, \ldots, x_{2^n}) & \mapsto& 
\displaystyle \frac{x_1+\cdots+x_{2^n}}{2^n}\ .
\end{array}
\end{align*}
Notice that $\triangle^{(n)}_P$ is a section of $\beta^{(n)}_P$. 
Any map $\tr$ commuting with the respective diagonal maps satisfies 
\[\tr = \beta^{(n)}_Q\circ\tr^{2^n}\circ\triangle^{(n)}_P\ . \]

Let us prove the statement by induction on the dimension $d$ of the polytopes $P$ and $Q$. It is obvious for $d=0$. For $d=1$, let us suppose that $P=Q=[0,1]$ and that $\Phi(0)=0$, $\Phi(1)=1$, without any loss of generality. 
Since the two polytopal subdivisions correspond bijectively under $\Phi$, the definition of the diagonal maps shows that $\vec v$ and $\vec w$ are oriented in the same direction.  By dichotomy, one can check that the $2^n$-tuple $\triangle^{(n)}_P\left(\frac{k}{2^n}\right)$ is made up of  $2^{n}-k$ zeros and $k$ ones. This induces the formula
\[\tr\left(\frac{k}{2^n}\right) 
= \beta^{(n)}_Q\circ\Phi^{2^n}\circ\triangle^{(n)}_P\left(\frac{k}{2^n}\right) 
= \frac{k}{2^n} \  . \]
By continuity,  the map $\tr$ is the identity of $[0, 1]$. 

Let us now suppose that the statement holds up to dimension $d-1$ and let $P$ and $Q$ be two polytopes of dimension $d$. Since the restriction of the diagonal map of $P$ to a face $F\in \La(P)$ is equal to the diagonal map of the face, i.e.  $\triangle_{(P, \vec v)}(z)=\triangle_{(F, \vec v)}(z)$, for any $z\in F$,  by \cref{Lemma:DiagOnFace}, the 
 induction hypothesis implies that the transition map $\tr$ exists and is uniquely defined on the $(d-1)$-skeleton of $P$. 
To study,  the interior of the top face of $P$, we consider the following filtration 
\[P(n)\coloneqq P\setminus \left(\sum_{k=0}^{2^n-1} \frac{k\bm P+(2^n-1-k)\tp P + P}{2^n} \right)\ ,\]
for $n> 0$. Notice that 
\[\bigcup_{n>0} P(n)=P \setminus [\bm P, \tp P]\ ,\]
where $[\bm P, \tp P]$ stands for the interval defined by the top and the bottom vertices of $P$.
\cref{cor:DIMENSION} shows that the faces appearing in the tight coherent subdivision 
 corresponding to the section 
$\triangle^{(n)}_P$ are of two kinds: $(\bm P, \ldots, \bm P, P, \tp P, \ldots, \tp P)$ or $(F_1, \ldots, F_{2^n})$, with 
$\codim F_i\geq 1$, for all $1\leq i \leq 2^n$. The images of the first ones under $\beta^{(n)}_P$ are equal to the sets excluded from $P$ in the definition of $P(n)$. Otherwise stated,  the image of any $z\in P(n)$ under the iterated diagonal map satisfies $\triangle^{(n)}_P(z)\in \left(\sk_{d-1}P\right)^{2^n}$, and so 
\[\tr(z) = \beta^{(n)}_Q\circ\big(\tr|_{\partial P}\big)^{2^n}\circ\triangle^{(n)}_P(z)\ . \]
The image of the transition map $\tr$ on the main axis $[\bm P, \tp P]$ is given by the same dichotomy argument as in the case $d=1$. In the end, this proves the uniqueness of the transition map. 

To show the existence of such a suitable transition map, we define it by the above-mentioned formulas. It  remains to prove the continuity at points $x\in [\bm P, \tp P]$ of the main axis. Let $\varepsilon>0$ and let us find 
 $\delta>0$ such that $\abs{x-y} <\delta$ implies $\abs{\tr (x)-\tr (y)}<\varepsilon$. 
We consider $n\in \ZP$ satisfying $\diam Q< 2^n\varepsilon$\ .
There are two cases to consider. 
\begin{enumerate}
\item When $x$ cannot be written as $(k\bm P+ (2^n-k)\tp P)/2^n$, with  $0\leq k\leq 2^n$, 
it is of the form $(k\bm P+ (2^n-1-k)\tp P+ x')/2^n$ for some $x'\in \mathring P$. 
In this case, we can take a small enough $\delta>0$ such that $\abs{x-y}<\delta$ 
implies that $y$ is of the form $(k\bm P+ (2^n-1-k)\tp P+ y')/2^n$, for some $y'\in \mathring P$.
This implies  $\triangle_P^{(n)}(x) = (\bm P, \ldots, \bm P, x', \tp P, \ldots, \tp P)$ and 
$\triangle_P^{(n)}(y) = (\bm P, \ldots, \bm P, y', \tp P, \ldots, \tp P)$, with the same number of $\bm P$ and $\tp P$, and so 
\[\abs{\tr(x)-\tr(y)} = \frac{\abs{\tr(x')-\tr(y')}}{2^n}\leq \frac{\diam Q}{2^n}< \varepsilon\ .\]

\item Otherwise, we can write $x$ as $(k\bm P+ (2^n-k)\tp P)/2^n$, with $0\leq k\leq 2^n$. 
We further divide into the following two cases and take the least $\delta$.
\begin{enumerate}
\item When $y\not\in P(n)$, we can take $\delta>0$ such that if $\abs{x -y}<\delta$, 
then $y$ is contained in $(k\bm P+ (2^n-1-k)\tp P+ P)/2^n$ or $((k-1)\bm P+ (2^n-k)\tp P+ P)/2^n$. 
This implies $\abs{\tr(x)-\tr(y)}<\varepsilon$ as above.

\item When $y\in P(n)$, observe that $\beta_Q^{(n)}\circ\left(\tr|_{\partial P}\right)^{2^n}\circ\triangle_P^{(n)}$ actually defines 
a continuous restriction of $\tr$ to the closed set
\[P\setminus \left(\sum_{k=0}^{2^n-1} \frac{k\bm P+(2^n-1-k)\tp P + \mathring P}{2^n} \right) \]
which contains both $x$ and $y$. Therefore we can choose $\delta$ which satisfies the condition.
\end{enumerate}
\end{enumerate}
\color{black}
\item This is straightforward from the above description. The inverse morphism in $\PolySub$ is the transition map 
$\tr_Q^P$. 
\end{enumerate}
\end{proof}

\begin{corollary}
Any two normally equivalent polytopes positively oriented by the same orientation vector $\vec v$ 
are isomorphic in the category $\PolySub$, with an isomorphism commuting with the diagonal maps. 
\end{corollary}

\begin{proof}
This is a direct corollary of \cref{Prop:NormalFan} and \cref{prop:Transition}. 
\end{proof}

This produces a stronger comparison between the diagonal maps of two Loday realizations associated to different weights than \cref{Cor:SameDiag}: the transition map $\tr=\tr_\omega^\theta : \K_\omega \to \K_\theta$ 
preserves homeomorphically the faces of the same type  and it commutes with the respective diagonals. Up to isomorphisms in the category $\PolySub$, the diagonal maps 
do not  depend on the orientation vector (with decreasing coordinates) nor on the weight.

\subsection{Operad structure}
We use the above results to endow the collection $\{\K_n\}_{n\geq 1}$ of Loday realizations (with standard weights) with an operad structure as follows. 

\begin{definition}[Operad structure]
For any $n,m\geq 1$ and any $1\leq i \leq m$, we define the \emph{partial composition map}  by 
\[
\vcenter{\hbox{
\begin{tikzcd}[column sep=1cm]
\circ_i\ : \ \K_m\times \K_n
\arrow[r,  "\tr\times \id"]
& \K_{(1,\ldots,n,\ldots,1)}\times \K_n 
 \arrow[r,hookrightarrow, "\Theta"]
&
\K_{n+m-1}\ ,
\end{tikzcd}
}}\]
where the last inclusion is given by the block permutation of the coordinates introduced in  the proof of \cref{prop:PropertiesKLoday}. 
\end{definition}

\begin{theorem}\label{thm:MainOperad}\leavevmode
\begin{enumerate}
\item The collection $\{K_n\}_{n\geq 1}$ together with the partial composition maps $\circ_i$ form a non-symmetric operad in the category $\PolySub$. 

\item The maps $\{\triangle_n : K_n \to K_n\times K_n\}_{n\geq 1}$ form a morphism of non-symmetric operads in the category $\PolySub$.
\end{enumerate}
\end{theorem}

\begin{proof}
By \cref{prop:Transition} and \cref{prop:DiaginPoly}, the various maps are morphisms in the category $\PolySub$. 

\begin{enumerate}
\item We need to prove the sequential and the parallel composition axioms of a non-symmetric operad, see \cite[Section~5.3.4]{LodayVallette12}.
The sequential composition axiom amounts to the commutativity of the following diagram 
\[
\vcenter{\hbox{
\begin{tikzcd}[column sep=1.2cm, row sep=0.9cm]
\K_l\times \K_m\times \K_n  
\arrow[r,  "\id\times \circ_j"] 
\arrow[d,  "\circ_i \times \id"]
& \K_l\times \K_{m+n-1} \arrow[d,  "\circ_i"] \\
\K_{l+m-1} \times \K_n   \arrow[r,  "\circ_{i+j-1}"]
& \K_{l+m+n-2}\ .
\end{tikzcd}
}}\]
Let us denote by $F$ the face of $\K_{l+m+n-2}$ labelled by the composite tree 
$c_l \circ_i (c_m \circ_j c_n)=(c_l \circ_i c_m) \circ_{i+j-1} c_n$. The two maps of this diagram have the same image equal to $F$. They both induce two cellular homeomorphisms $\K_l\times \K_m\times \K_n  \to F$ which meet the requirements of \cref{prop:Transition} by  \cref{prop:PropertiesKLoday} and by the fact that the transition map $\tr$ and the isomorphism $\Theta$ commute with the diagonal maps. So they are equal by Point~(1) of \cref{prop:Transition}. 
The parallel composition axiom is proved in the same way. 

\item This statement means that the partial composition maps commute with the diagonal maps, which is the case since the maps $\tr$ and $\Theta$ do. 
\end{enumerate}
\end{proof}

Under the cellular chain functor, we recover the classical differential graded non-symmetric operad 
$\mathcal{A}_\infty$ encoding homotopy associative algebras \cite{Stasheff63}, see also \cite[Chapter~9]{LodayVallette12}. 
To understand the induced diagonal on the differential graded level, we need the following magical formula describing its cellular structure. 

\section{The magical formula}\label{sec:MagicalFormula}

\begin{theorem}[Magical formula]\label{thm:MagicFormula}
For any Loday realization of the associahedra, the approximation of the diagonal satisfies 
\[{\im \triangle_n 
=\bigcup_{\tp F\leq \bm G\atop \dim F+\dim G=n-2} F\times G}\ .\] 
\begin{figure}[h]
\[
\begin{tikzpicture}[scale=0.8]
\draw[thick] (-5,0)--(0,4)--(5,0)--(2.5,-4)--(-2.5,-4)--cycle;
\draw[fill, opacity=0.12] (-5,0)--(0,4)--(5,0)--(2.5,-4)--(-2.5,-4)--cycle;

\draw (-2.5,2)--(0,0)--(-1.5,-2)--(-3.75,-2);
\draw (2.5,2)--(0,0)--(1.5,-2)--(3.75,-2);
\draw (-1.5,-2)--(0,-4)--(1.5,-2);

\draw (-2.7,-0.25) node {$\TreeLL\ \times\TreeC$};
\draw (2.7,-0.25) node {$\TreeC\ \times\TreeRR$};
\draw (0,-2.2) node {$\TreeAL\ \times\TreeRA$};
\draw (0,2) node {$\TreeLA\ \times\TreeAR$};
\draw (-2.25,-3) node {$\TreeAL\ \times\TreeCA$};
\draw (2.15,-3) node {$\TreeCA\ \times\TreeRA$};

\draw (-5,0) node[left] {$\TreeLL$};
\draw (5,0) node[right] {$\TreeRR$};
\draw (0,4) node[above] {$\TreeCC$};
\draw (-2.5,-4) node[below left] {$\TreeLR$};
\draw (2.5,-4) node[below right] {$\TreeRL$};
\draw (-4,-2) node[left] {$\TreeAL$};
\draw (4,-2) node[right] {$\TreeRA$};
\draw (-2.5,2) node[above left] {$\TreeLA$};
\draw (2.5,2) node[above right] {$\TreeAR$};
\draw (0,-4) node[below] {$\TreeCA$};
\end{tikzpicture}
\]
\caption{The polytopal subdivision $\beta\left(\sF_4\right)$ of $\K_4$.}
\end{figure}
\end{theorem}

The pairs of faces appearing on the right-hand side of the magical formula are called \emph{matching pairs}. 
In other words, the tight coherent subdivision 
$\sF_n=\left\{ (F,G) \mid \tp F\leq \bm G,\, \dim F+\dim G=n-2\right\}$, made up of matching pairs, gives a polytopal subdivision of the associahedra under $\beta$.
Applying the cellular chain functor, we recover the differential graded diagonal given in \cite{SaneblidzeUmble04, MarklShnider06}. Notice that neither the pointwise version nor the cellular version of the diagonal map $\triangle_n$ can be coassociative by \cite[Theorem~13]{MarklShnider06}. 

\subsection{First step: $\im \triangle_n \subset \bigcup F\times G$}\label{Subsec:Subset}
We prove this property more generally for any \emph{product} $P \coloneqq K_{\omega_1}\times \cdots \times K_{\omega_k}$ of Loday realizations of the associahedra. Recall that 
$P\subset \RR^{N}$, where $N\coloneq n_1+\cdots+n_k-k$, and $d\coloneqq \dim P = n_1+\cdots+n_k-2k$. 
The set $\Ve(P)$ of vertices of $P$ coincides with 
$\Tam{n_1}\times \cdots \times \Tam{n_k}$. 
By \cref{prop:OrientationVector}, any vector 
 $\vec v= (v_1, v_2, \ldots, v_N)$ satisfying 
\[v_1>\cdots>v_{n_1-1}\ ,\ \  v_{n_1}>\cdots>v_{n_1+n_2-2}\ ,\quad  \ldots\ , \quad   v_{n_1+\cdots+n_{k-1}-k+2}>\cdots>v_{N}\  \]
makes $(P, \vec v)$ into a positively oriented polytope with $1$-skeleton isomorphic to the product of Tamari lattices. 

\medskip

We consider the  map $L=(L_i)_{1\leq i\leq m-2} : \Tam{m}\to \Bool{m-2}$ 
to the Boolean lattice
defined by 
\begin{align*}
L_i(t)\coloneqq\begin{cases}
0& \text{if the $(i+1)$-th leaf is left-leaning $\swarrow$\ ,}\\
1& \text{if the $(i+1)$-th leaf is right-leaning $\searrow$}\ .
\end{cases}
\end{align*}
We extend it to a map $L=(L_i)_{1\leq i \leq d}: \Ve(P)\to \Bool{d}$. We consider the collection of vectors 
$\vec e_1, \ldots, \vec e_d$ in $\RR^n$ defined by 
$\vec e_i\coloneqq (0, \ldots, 1, -1, \ldots, 0)$, where $1$ is in the $k$-th place for $k\neq n_1+\cdots+n_j-j$.

\begin{lemma}\label{lemma:propertiesofL}
The following properties are satisfied:
\begin{enumerate}
\item $s\leq t$  $\Rightarrow L(s)\leq L(t)$ ,

\item $\ang{\vec e_i, \vec v} >0$, for $1\leq i \leq d$ , 

\item each edge  connecting $s\in L^{-1}_i (0)$ and $t\in L^{-1}_i(1)$ is parallel to $\vec e_i$ 

\item any fiber $L^{-1}(b)$, for $b\in  \Bool{d}$,  is contained in a facet of $P$ . 
\end{enumerate}
\end{lemma}

\begin{proof}\leavevmode
\begin{enumerate}
\item When we switch a pair of successive left and right-leaning edges to a pair of successive right and left-leaning edges, either it does not change the orientation of the leaves or it just changes one left-leaning leaf into a right-leaning leaf. 

\item  This is  straightforward from the definition of $\vec v$. 

\item It is enough to prove it on one polytope $K_\omega$. 
In this case, $t$ is a planar binary tree  obtained from a planar binary tree $s$ by switching the $(i+1)^{\rm th}$ leaf from left-leaning to right-leaning, which implies 
\[
\overrightarrow{M(s, \omega)M(t, \omega)}=\omega_i\omega_{i+2}\vec e_i\ .
\]
\item Reading the sequence $(1, b_1, \ldots, b_{n_1-2}, 0)$ from left to right, there is at least one occurence of $(1,0)$, say at place $i$ and $i+1$. Every face labeled by a forest of trees $t$ satisfying $L(t)=b$ lies in the facet labeled by $\big(c_{n_1-1}\circ_i c_2, c_{n_2}, \ldots, c_{n_k}\big)$.
\end{enumerate}
\end{proof}

\begin{lemma}\label{lem:FACE}
Let $F$ be a face of $P$ which contains a vertex $s$ such that $L_i(s) = 1$. For any $x\in \mathring F$, there exists $\varepsilon>0$ such that $x- \varepsilon \vec e_i\in P$.
\end{lemma}

\begin{proof}
It is enough to perform the proof for one Loday realization $P=\K_\omega$ that we endow with the linear form 
$\psi(x)\coloneqq \ang{x, \vec e_i}$.
We claim that, for the projection $\id : P \to P$, the associated subcomplexes of lower and upper faces, introduced in \cref{subsec:PolySubDiv}, are given by 
\[\Ve\Big(\La^{\downarrow}\big(\tilde P\big)\Big)=\left\{
 M(t,\omega)\mid L_i(t)=0 
\right\}
\quad 
\text{and}
\quad 
\Ve\Big(\La^{\uparrow}\big(\tilde P\big)\Big)=\left\{
 M(t,\omega)\mid L_i(t)=1
\right\}\ .
\]
The intersection $L_i^{-1}(0)\cap \La^{\downarrow}\big(\tilde P\big)$ 
is nonempty, since it contains the  vertex $\bm P$. Suppose that the polytopal complex 
$\La^{\downarrow}\big(\tilde P\big)$ contains a vertex living in $L_i^{-1}(1)$. By the connectivity of $\La^{\downarrow}\big(\tilde P\big)$, it admits an edge with vertices $s\in L_i^{-1}(0)$ and $t\in L_i^{-1}(1)$ . 
Point~(3) of \cref{lemma:propertiesofL} says that such an edge is parallel to $\vec e_i$, which contradicts the minimality of  $\La^{\downarrow}\big(\tilde P\big)$. 
This proves the inclusion 
$\Ve\big(\La^{\downarrow}\big(\tilde P\big)\big)
\subset
\left\{
 M(t,\omega)\mid L_i(t)=0 
\right\}$.
The opposite inclusion is proved by the same argument as for 
$\La^{\uparrow}\big(\tilde P\big)$
together with the fact that 
the union $\La^{\downarrow}\big(\tilde P\big)\cup \La^{\uparrow}\big(\tilde P\big)$ contains all the vertices of $P$. 
If not, a vertex $x$, which is not contained in it, is neither maximal nor minimal with respect to $\vec e_i$, which contradicts the fact that $x$ is an extremal point.
This characterization of the polytopal complex of lower faces shows that any face $F$ of $P$ containing a vertex $s$ such that $L_i(s) = 1$ satisfies 
$ F\not \subset  \big\vert\La^{\downarrow}\big(\tilde P\big)\big\vert$, and thus 
$\mathring F \cap \big\vert\La^{\downarrow}\big(\tilde P\big)\big\vert=\emptyset$, which concludes the proof. 
\end{proof}

\begin{proposition}\label{prop:FirstExclusion}
Let $F$ and $G$ be two faces of $P$ of matching dimensions, i.e. $\dim F+\dim G=d$. We consider 
$s\coloneqq \tp F$ and $t\coloneqq \bm G$.
When $L(s)\not\leq L(t)$, we have $\mathring{F}\times \mathring{G}\cap \im\triangle = \emptyset$.
\end{proposition}

\begin{proof}
When $L(s)\not\leq L(t)$, there exists $i$ such that $L_i(s) = 1$ and $L_i(t) = 0$. 
By \cref{lem:FACE}, 
 for every $x\in \mathring{F}$ and $y\in \mathring{G}$, there exists $\varepsilon>0$ such that $(x-\varepsilon \vec e_i, y+\varepsilon \vec e_i) \in P\times P$. 
Suppose now that  there exists $(x,y)\in \mathring{F}\times \mathring{G} \cap \im \triangle$. 
Since $(x+y)/2 = ((x-\varepsilon\vec e_i)+(y+\varepsilon\vec e_i))/2$, the two points $(x, y)$ and $(x-\varepsilon\vec e_i, y+\varepsilon\vec e_i)$ lie in the same fiber of $\beta$. As we saw in the proof of
\cref{prop:DiaginPoly}, the point $(x,y)$ minimizes $\ang{x, \vec v}$ in the fiber of $\beta$.
However, the computation 
\[ 
\ang{\vec v, x-\varepsilon \vec e_i}=\ang{\vec v, x}-\varepsilon\underbrace{\ang{\vec v, \vec e_i}}_{>0}
<\ang{\vec v, x}
\]
violates the definition of $\triangle$. 
\end{proof}

\cref{prop:FirstExclusion} excludes 
faces $F\times G$ of matching dimensions with  $L(s)\not\leq L(t)$
from the image of $\triangle$. 
Notice that, by Point~(1) of \cref{lemma:propertiesofL}, $L(s)\not\leq L(t)$ implies $s\not \leq t$.
In order to exclude the remaining case when $s\not\leq t$ and $L(s)\leq L(t)$, we prepare the following two lemmas.

\begin{lemma}\label{lem:UniqueFG}
Let $t\in  \Tam{n_1}\times \cdots \times \Tam{n_k}$ be a forest of  planar binary trees with a total of $r_t+k$ right-leaning leaves and $l_t+k$ left-leaning leaves. There exists a unique maximal (with respect to inclusion) face $F_t$ (resp. $G_t$) of $P$  such that $\tp F = t$ (resp. $\bm G = t$). The dimensions of these faces are given by $\dim F_t =l_t$ and $\dim G_t = r_t$.
\end{lemma}

\begin{proof}
The cell $F_t$ (resp. $G_t$) can be obtained by collapsing all the left-leaning (resp. right-leaning) internal edges of all the  trees of the forest $t$. One can see that, for any forest of planar binary trees, the number of left-leaning (resp. right-leaning) internal edges is equal to the number of right-leaning (resp. left-leaning) leaves minus $k$. 
\end{proof}

\begin{lemma}\label{lemma:Matching}
Let $F,G\subset P$ be a  pair of faces of matching dimensions. 
When  $s\coloneqq \tp F$ and $t\coloneqq \bm G$ satisfy $L(s)\leq L(t)$, then $(F, G) = (F_s, G_t)$ and $L(s)=L(t)$. 
\end{lemma}

\begin{proof}
By definition,  $F\subset F_s$ and $G\subset G_t$, and thus $\dim F + \dim G\leq l_s + r_t\leq d$ by \cref{lem:UniqueFG}. 
The top dimension assumption  $\dim F+\dim G=d$ allows us to conclude that 
$\dim F=l_s$, $\dim G=r_t$,  $F=F_s$, $G=G_t$, and $L(s)=L(t)$.
\end{proof}

We can now conclude Step 1: we prove by induction on the dimension $d$ of $P$ that any pair $F, G\subset P$ of  faces of matching dimensions with $s\not\leq t$ and $L(s)\leq L(t)$ satisfies 
${F}\times {G}\not\subset \im\triangle$.
This is straightforward to check in dimensions $d=0$ and $d=1$. 
In dimension $d$, let us suppose, to the contrary, that there exists such a pair $F,G$ of faces satisfying ${F}\times {G}\subset \im\triangle_{(P, \vec v)}$.
\cref{lemma:Matching} implies $L(s)=L(t)$.
In this case, both points $s$ and $t$ lie in a common facet $Q$ of $P$ by Point~(4) of \cref{lemma:propertiesofL}.
By \cref{Lemma:DiagOnFace}, the induced diagonal  $\triangle_{(Q,\vec v)}$ on $Q$  is the restriction of ${\triangle_{(P, \vec v)}}$. 
Let us consider $F'\coloneqq F\cap Q$ and $G'\coloneqq G\cap Q$. If $\dim F'+\dim G'> \dim Q$, then 
$F'\times G'\not \subset \triangle_{(Q,\vec v)}$ by \cref{cor:DIMENSION}. When $\dim F'+\dim G'\leq \dim Q$, we consider any pair of faces $F'\subset F''$ and $G'\subset G''$ of $Q$ of matching dimensions: $\dim F''+\dim G''=\dim Q$. 
They satisfy $r\coloneqq\tp F''\geq \tp F'=s$ and $u \coloneqq\bm G''\leq \bm G'=t$, which implies $r \not \leq u$. 

We claim that ${F''}\times {G''}\not\subset \im\triangle_{(Q,\vec v)}$.
Since $Q$ is a  facet of $P$, it is of type 
$Q'\coloneqq K_{\bar \omega_1}\times K_{\tilde \omega_1}\times K_{\omega_2} \times \cdots \times K_{\omega_k}\cong Q$, 
under the isomorphism $\Theta : \RR^N \cong \RR^N$ given by the  block permutation of coordinates described in the proof of Point~(5) of \cref{prop:PropertiesKLoday}. The image of the orientation vector $\vec v=(v_1, \ldots, v_{n_1-1}, v_{n_1},  \ldots)$ 
under the inverse permutation of coordinates  is equal to 
\[\vec v':=\Theta^{-1}(\vec v)=(v_1, \ldots, v_{i-1}, v_{i+m-1}, \ldots, v_{n_1-1}, v_i, \ldots, v_{i+m-2}, v_{n_1}, \ldots, )\ ,\]
so it well-orients $Q'$. 
 Therefore the isomorphism $\Theta$ intertwines the two diagonals $\triangle_{(Q,\vec v)}$ and $\triangle_{(Q',\vec v')}$. Let us denote by $r=(r_1, \ldots, r_k)$ and $u=(u_1, \ldots, u_k)$ the various planar trees composing the two forests. Under the notation of the proof of 
 Point~(5) of \cref{prop:PropertiesKLoday},
  one can write the two planar trees 
 $r_1=\bar r_1\circ_i \tilde r_1$ and $u_1=\bar u_1\circ_i \tilde u_1$. This gives 
 \[r'\coloneqq \tp \Theta^{-1}(F'')=(\bar r_1, \tilde r_1, r_2, \ldots, r_k)
\qquad \text{and} \qquad  u' \coloneqq \bm  \Theta^{-1}(G'')=(\bar u_1, \tilde u_1, u_2, \ldots, u_k) \ .
 \]
We have $r'\not\leq u'$. Indeed, the condition $r\not \leq u$ implies that there exists $1\leq j\leq k$ such that $r_j \not \leq  u_j$. If $j\neq 1$, then automatically $r'\not\leq u'$. If $j=1$, then either $\bar r_1 \not \leq \bar u_1$ or 
$\tilde r_1 \not \leq \tilde u_1$, since $\circ_i$ preserves the Tamari order, and again $r'\not\leq u'$. 
If  $L(r')\not\leq L(u')$,
 we conclude with \cref{prop:FirstExclusion}, otherwise we conclude our claim with the induction hypothesis. 
Finally, notice that any cell appearing in $\im \triangle$ is contained in a product of cells of matching dimensions. The above argument shows that $F'\times G'$ cannot appear in $\im \triangle_{(Q,\vec v)}$. This concludes the first step. 

\subsection{Second step: $\im\triangle_n\supset \bigcup F\times G$}
In this section, we prove that  every matching pair $(F, G)$ of $\K_n$ satisfies $F\times G\subset \im\triangle_n$. 
By Point (1) of \cref{lemma:propertiesofL} and by \cref{lemma:Matching} such pairs are of type 
$(F_s, G_t)$ with $s\leq t$ and $L(s)=L(t)$.

\begin{proposition}\label{prop:FtGt}
For every $t\in \Tam{n}$, we have $F_t\times G_t\subset \im \triangle_n$. 
\end{proposition}

\begin{proof}
By the pointwise formula $(t, t)\in \im\triangle_n$. But $F_t\times G_t$ is the only cell coming from a matching pair that can contain $(t, t)$ by \cref{Subsec:Subset}.
\end{proof}

\begin{lemma}\label{lem:FacetGuGt}
Let  $t\prec u$ be an edge of $\K_n$ satisfying $L(t) = L(u)$. 
\begin{enumerate}
\item The cell $G_t\cap G_u$ is a facet of $G_t$ and $G_u$. 

\item The cell $G_t\cap G_u$ is not of type $G_v$, for $v\in \Tam{n}$.

\item If $G_t\cap G_u$ is a facet of $G_v$, then $v=t$ or $v=u$.
\end{enumerate}
\end{lemma}

\begin{proof}\leavevmode
\begin{enumerate}
\item Let  the planar binary tree $u$ be obtained from $t$ by switching 
a pair $(f,e)$ of successive left and right-leaning edges
to a pair $(f', e')$ of successive 
 right and left-leaning edges:
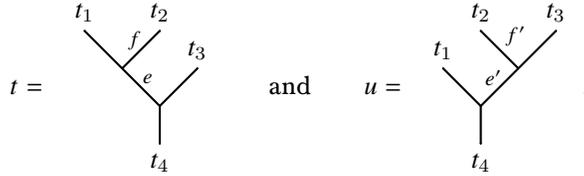
\[
t=\ {\vcenter{\hbox{
\begin{tikzpicture}[yscale=0.5,xscale=0.5]
\draw[thick] (0,-1)--(0,0) -- (-2,2);
\draw[thick] (-1,1)--(0,2) ;
\draw[thick] (0,0)--(1,1) ;
\draw [fill] (0,-1) circle [radius=0.014];
\draw [fill] (-2,2) circle [radius=0.014];
\draw [fill] (1,1) circle [radius=0.014];
\draw [fill] (0,2) circle [radius=0.014];

\draw (-2,2) node[above] {$t_1$}; 
\draw (0,2) node[above] {$t_2$}; 
\draw (1,1) node[above] {$t_3$}; 
\draw (0,-1) node[below] {$t_4$}; 
\draw (-0.7,0.35) node[above right] {\scalebox{0.8}{$e$}}; 
\draw (-0.25,1.25) node[above left] {\scalebox{0.8}{$f$}}; 
\end{tikzpicture}}}}
\qquad \text{and} \qquad 
u=\ {\vcenter{\hbox{
\begin{tikzpicture}[yscale=0.5,xscale=0.5]
\draw[thick] (0,-1)--(0,0) -- (2,2);
\draw[thick] (1,1)--(0,2) ;
\draw[thick] (0,0)--(-1,1) ;
\draw [fill] (0,-1) circle [radius=0.014];
\draw [fill] (2,2) circle [radius=0.014];
\draw [fill] (-1,1) circle [radius=0.014];
\draw [fill] (0,2) circle [radius=0.014];

\draw (2,2) node[above] {$t_3$}; 
\draw (0,2) node[above] {$t_2$}; 
\draw (-1,1) node[above] {$t_1$}; 
\draw (0,-1) node[below] {$t_4$}; 
\draw (0.8,0.3) node[above left] {\scalebox{0.8}{$e'$}}; 
\draw (0.4,1.35) node[above right] {\scalebox{0.8}{$f'$}}; 
\end{tikzpicture}}}}\ .
\]
By \cref{lem:UniqueFG}, the cell $G_t\cap G_u$ corresponds to a tree $s$ obtained by collapsing all the right-leaning internal edges of $u$ except $f'$. 
The cell $G_u$ corresponds to the tree 
$s / f'$ obtained from $s$ by contracting the edge $f'$.
The cell $G_t$ corresponds to the tree $s/ e'$.

\item Notice first that a cell $G$ is of type $G_v$, for $v\in \Tam{n}$ if and only if it is labeled by a planar tree 
which cannot be factored by $\circ_1$. 
This is not the case of the tree $s$, since the edge $e$ exhibits a factorization of $s$ by $\circ_1$.

\item In this case, the labelling tree $r$ of $G_v$ is obtained from $s$ by contracting one internal edge. Since $r$ should not contain any internal edge which exhibits a factorization by $\circ_1$, it can only be obtained in two ways: by contracting $e'$ or $f'$. In the former case $v=t$ and in the latter case $v=u$. 
\end{enumerate}
\end{proof}

\begin{proposition}
If $s\leq t$ and $L(s) = L(t)$ then $F_s\times G_t\subset \im\triangle_n$.
\end{proposition}

\begin{proof}
We fix $s$ and we consider the sub-lattice $\Tam{n}^{\geq s}$ of elements greater than $s$. 
Let us prove by induction on the elements $t$ of $\Tam{n}^{\geq s}\cap L^{-1}(L(s))$ that 
$F_s\times G_t\subset \im\triangle_n$.
The base case is $s=t$ and this is done in the above \cref{prop:FtGt}.
Suppose the conclusion holds for $t$ and let $t\prec u$ be an edge satisfying $L(t)= L(u)$. 
  Then $H \coloneqq F_s\times G_t\cap F_s\times G_u = F_s\times (G_t\cap G_u)$ is a facet of $F_s\times G_t$
 by Point~(1) of \cref{lem:FacetGuGt} and it does not project into the boundary of $\K_n$ via $\beta$.
 Since $F_s\times G_t$ lies in $\im\triangle_n$, it is attached along  $H$ to a cell 
$F_r\times G_v$ by \cref{Subsec:Subset}. 
We claim that $r=s$, $v=u$. 
Since $H=F_s\times (G_t\cap G_u)$ is a facet of $F_r\times G_v$, there are two options: either $F_s=F_r$ or $G_t\cap G_u=G_v$. The latter case is contradicted by Point~(2) of \cref{lem:FacetGuGt}.  
Therefore, $G_t\cap G_u$ is a facet of $G_v$, and we conclude by 
Point~(3) of \cref{lem:FacetGuGt}.  
\end{proof}

This concludes the proof of Step 2. 

\bibliographystyle{amsplain}
\bibliography{bib}

\end{document}